\documentclass[a4paper,11pt]{article}
\usepackage[utf8x]{inputenc}

\usepackage{xfrac}
\usepackage{parskip}
\usepackage{a4wide}
\usepackage{amssymb}
\usepackage{amsmath}
\usepackage{mathtools}
\usepackage{amsthm}
\usepackage[mathscr]{euscript}
\usepackage[normalem]{ulem}
\usepackage[colorlinks=true]{hyperref}
\usepackage[usenames,dvipsnames]{xcolor}
\usepackage{appendix}
\usepackage{tikz}
\usetikzlibrary{arrows,angles}
\usetikzlibrary{matrix}
\usetikzlibrary{decorations}
\usetikzlibrary{arrows,calc,shapes,decorations.pathreplacing}
\usepackage{tikz-cd}
\usepackage{todonotes}
\usepackage{caption}
\hypersetup{hidelinks}

\numberwithin{equation}{section}
\usepackage{graphicx}

\usepackage{nicematrix}
\usepackage[xcolor]{changebar}
\usepackage{enumitem}
\usepackage{mathrsfs}
\usepackage{yfonts}
\usepackage{caption}
\usepackage{dsfont}
\usepackage{centernot}

\newcommand\restr[2]{{
  \left.\kern-\nulldelimiterspace 
  #1 
  \vphantom{\big|} 
  \right|_{#2} 
  }}
\newcommand{\eps}{\varepsilon}
\newcommand{\R}{\mathbb{R}}
\DeclareMathOperator{\CAT}{CAT}
\newcommand{\N}{\mathbb{N}}
\newcommand{\dotto}{\, \cdot \,}

\usepackage{accents}

\usepackage{thmtools}

\declaretheorem[
	name=Theorem,
	numberwithin=section
	]{thm}
\declaretheorem[
	name=Lemma,
	numberlike=thm,
	]{lem}
\declaretheorem[
	name=Lemma,
	numbered=no,
	]{lem*}
\declaretheorem[
	name=Proposition,
	sibling=thm,
	]{prop}
\declaretheorem[
	name=Corollary,
	sibling=thm,
	]{cor}

\declaretheorem[
	name=Definition,
	style=definition,
        numberlike=thm,
	]{defin}

\declaretheorem[
	name=Remark,
	numberlike=thm,
        style=definition,
	]{rem}

\usepackage{pgfplots}
\pgfplotsset{compat=1.15}
\usepackage{mathrsfs}
\usetikzlibrary{arrows}

\title{A synthetic Lorentzian Cartan-Hadamard theorem}
\author{Darius Erös\thanks{{\tt darius.eroes@univie.ac.at}, Faculty of Mathematics, University of Vienna, Austria.} \ and Sebastian Gieger\thanks{{\tt sebastian.gieger@univie.ac.at}, Faculty of Mathematics, University of Vienna, Austria.}}
\date{\today}

\begin{document}

\definecolor{rvwvcq}{rgb}{0.08235294117647059,0.396078431372549,0.7529411764705882}
\definecolor{wrwrwr}{rgb}{0.3803921568627451,0.3803921568627451,0.3803921568627451}
\definecolor{sexdts}{rgb}{0.1803921568627451,0.49019607843137253,0.19607843137254902}

\maketitle
\vspace{-0.5em}
\begin{abstract}
We formulate and prove a synthetic Lorentzian Cartan-Hadamard theorem. This result both transfers the corresponding statement for locally convex metric spaces established by S.\ Alexander and R.\ Bishop to the Lorentzian setting, and simultaneously generalizes the smooth Lorentzian theorem discussed by J.\ Beem and P.\ Ehrlich to the recently established framework of synthetic Lorentzian geometry. Our approach is based on an appropriate notion of local concavity for Lorentzian (pre-)length spaces, which allows us to establish existence and uniqueness of timelike geodesics between any pair of timelike related points under the additional assumptions of global hyperbolicity and future one-connectedness. We also provide a globalization result for our notion of concavity in the setting of Lorentzian length spaces, and apply our results to obtain a globalization statement for nonnegative upper timelike curvature bounds.

\medskip

\end{abstract}

\renewcommand{\baselinestretch}{0.75}\normalsize
\tableofcontents
\renewcommand{\baselinestretch}{1.0}\normalsize
\newpage
\section{Introduction}
The Cartan-Hadamard theorem, as a typical local-to-global result, describes the structure of complete spaces of non-positive sectional curvature. In its classical form, it reads:

\begin{thm}[Cartan-Hadamard]
    Let $(M,g)$ be a complete and connected $n$-dimensional Riemannian manifold with non-positive sectional curvature. Then its universal cover is diffeomorphic to $\R^n$.
\end{thm}

In the 1980s, a far-reaching generalization to the class of complete and locally convex geodesic spaces was proposed by Mikhail Gromov, see \cite{GromovI}, \cite{GromovII}. To see how (local) convexity enters the picture, observe that any metric $d$ of a $\CAT(0)$ space $(X,d)$ is \textit{convex} in the sense that
\begin{equation*}
    d(\gamma_1(t), \gamma_2(t)) \leq (1-t) d(\gamma_1(0), \gamma_2(0)) + td(\gamma_1(1), \gamma_2(1))
\end{equation*}
for any two minimizing geodesics $\gamma_1,\gamma_2$ and all $t \in [0,1]$ by a simple application of the intercept theorem in the Euclidean plane. In fact, for a Riemannian manifold $(M,g)$ with its induced distance $d_g$, the convexity of $d_g$ is \textit{equivalent} to an upper sectional curvature bound on $g$ by zero, cf.\ \cite[Rmk.\ 1A.8]{bridson-haefliger}. However, for the more general class of Finsler manifolds, convexity of the distance metric is a \textit{strictly weaker} restriction on the geometry of the space, see \cite{LytchakIvanov}. This can already be observed at the level of linear spaces: Any normed space satisfying the $\CAT(0)$ inequality is automatically a (pre-)Hilbert space \cite[Prop.\ 1.14]{bridson-haefliger}.

This convexity assumption, first studied by Busemann and therefore also called \textit{Busemann convexity}, lies at the heart of a number of important theorems on the geometry of $\CAT(0)$ spaces. 
In particular, it allowed S.\ Alexander and R.\ Bishop to prove the aforementioned generalization of the Cartan-Hadamard theorem stated by M.\ Gromov \cite{AlexanderBishop}, while W.\ Ballmann provided a proof for the locally compact case \cite{BallmannPaper}. Following \cite[Thm.\ 4.1]{bridson-haefliger}, a generalized Cartan-Hadamard theorem reads:

\begin{thm}[Metric Cartan-Hadamard]\label{thm:MetricCH}
    Let $(X,d)$ be a complete and connected metric space. If $d$ is locally convex, then any pair of points in its universal cover $\tilde X$ is connected by a unique geodesic, and these geodesics vary continuously with their endpoints.
\end{thm}

In addition to these developments, a version of the Cartan-Hadamard theorem has been put forward, which holds for smooth Lorentzian manifolds, see \cite[Ch.\ 11.3]{BeemEhrlich}. In this setting, one is concerned with timelike curves, and so statements about the geometry of the universal cover cannot be expected. Instead, one assumes the Lorentzian manifold to be future one-connected, meaning that any two timelike curves connecting the same endpoints are homotopic through a family of timelike curves. Moreover, the completeness assumption is substituted by global hyperbolicity, which then finally leads to a smooth Lorentzian Cartan-Hadamard theorem of the following form, cf.\ \cite[Prop.\ 11.13,  Thm. 11.16]{BeemEhrlich}:

\begin{thm}[Smooth Lorentzian Cartan-Hadamard]\label{thm:SmoothLorCH}
    Let $(M,g)$ be a future one-connected globally hyperbolic spacetime with non-negative timelike sectional curvature. Then, given any $p,q \in M$ with $p \ll q$, there is exactly one future-directed timelike geodesic from $p$ to $q$.
\end{thm}

It is the purpose of this article to achieve a synthesis between \autoref{thm:MetricCH} and \autoref{thm:SmoothLorCH}, i.e., to state and prove a \emph{synthetic Lorentzian} Cartan-Hadamard theorem. This is in line with a broader effort to extend general relativity beyond the (smooth) spacetime setting, which has seen a surge of activity in recent years with the introduction of a synthetic framework for Lorentzian geometry. The core notion therein is that of a Lorentzian length space \cite{LLS}, which is a Lorentzian analog of a metric length space, whose geometry is captured by its (intrinsic) time separation function. Our main result is:

\begin{thm}[Synthetic Lorentzian Cartan-Hadamard]
    Let $X$ be a globally hyperbolic regular Lorentzian length space. If $X$ is locally concave and future one-connected, then every pair of timelike related points is connected by a unique timelike geodesic, and these geodesics vary continuously with their endpoints. 
\end{thm}

In fact, we will provide a slight generalization of the above to the setting of so-called Lorentzian \emph{pre-}length spaces, see \autoref{großesThm}. Moreover, as in the metric setting, this result will also enable us to globalize a upper curvature bound by zero.

This article is structured as follows: In \autoref{sec:Preliminaries}, we start by recalling basic definitions pertaining to the theory of Lorentzian length spaces, and fix notations and conventions before introducing a precise definition of local concavity. This definition will be compared to other notions of curvature bounds, and we will use it to show uniqueness of geodesics within comparison neighborhoods (\autoref{uniquemaxlocal}). \\
In \autoref{sec:LocalExistence}, this result will then be extended to an entire neighborhood of a given geodesic (\autoref{großesThm}). This is essentially a synthetic analog of the fact that in a Lorentzian manifold with timelike sectional curvature bounded from above by $0$, timelike geodesics do not encounter conjugate points \cite[10.20.\ Cor.]{oneil}.\\
Finally, in \autoref{sec:Globalization}, we will provide various applications of \autoref{großesThm}. First, we will show that geodesics within a comparison neighborhood vary continuously with their endpoints (\autoref{contvargeod}). Then we will establish a technical lemma guaranteeing that, under suitable assumptions, whenever two comparison neighborhoods are connected by a timelike geodesic, we can define continuously varying geodesics between any two points in these neighborhoods in a unique way (\autoref{globlem}). This will be the crucial step in showing that under the same assumptions, any two timelike related points in a locally concave Lorentzian pre-length space can be connected by a timelike geodesic. Furthermore, if we assume our space to be future one-connected, these geodesics are even unique and vary continuously with their endpoints, which is precisely the Lorentzian version of the Cartan-Hadamard Theorem we set out to prove. Finally, under similar assumptions, we provide globalization results for our notion of concavity as well as for upper curvature bounds as defined in \cite{curvature}. 

\vspace{-0.5em}
\section{Preliminaries}\label{sec:Preliminaries}
\vspace{-0.5em}
Over the last few years, several frameworks have been developed to allow for a synthetic conceptualization of Lorentzian geometry in the spirit of metric geometry (cf.\ \cite{LLS}, \cite{Olaf}, \cite{BraunMcCann}, \cite{MinguzziSuhr}). For the purposes of this article, we will assume basic familiarity with the theory of Lorentzian length spaces as introduced in \cite{LLS}. Nonetheless, in order to avoid ambiguity, we will briefly recall its central definitions in this section. By a \textit{Lorentzian pre-length space} we mean a $5$-tuple $(X,d,\ll,\leq,\tau)$ consisting of a metric space $(X,d)$, a timelike relation $\ll$ included in a causal relation $\leq$ and a non-negative, lower semicontinuous function $\tau$, satisfying
$$\tau(x,z)\geq\tau(x,y)+\tau(y,z)\hspace{1cm}\text{for all }x\leq y\leq z\in X$$
and
$$\tau(x,y)>0\hspace{0.5cm}\text{if and only if}\hspace{0.5cm}x\ll y.$$
The function $\tau \colon X \times X \rightarrow[0,+\infty]$ is called the \textit{time separation function} of the Lorentzian pre-length space $X$. Throughout this article, we will often impose additional regularity assumptions on $X$, such as being locally causally closed, timelike or causally path-connected, and (regularly) localizable. Precise definitions of these terms can be found in \cite{LLS}. 

A Lorentzian pre-length space $X$ is called \textit{strongly causal} if the timelike diamonds $I(x,y):=I^+(x)\cap I^-(y)$ form a subbasis of the topology induced by $d$. It will be called \textit{globally hyperbolic} if all causal diamonds $J(x,y):=J^+(x)\cap J^-(y)$ are compact and for every compact subset $K\subseteq X$ there exists a constant $C(K)$ such that the $d$-arc length of all causal curves contained in $K$ is bounded by $C(K)$. Here, a non-constant curve $\gamma \colon I \rightarrow X$, defined on some interval $I \subseteq \R$, is said to be \emph{future-directed causal} if $\gamma$ is locally Lipschitz (w.r.t.\ $d$) and $\gamma(t_1) \leq \gamma(t_2)$, whenever $t_1 < t_2$. Future- and past-directed timelike curves are introduced analogously. Note that global hyperbolicity only implies strong causality under additional assumptions (see \cite[Thm.\ 3.26]{LLS}).

We define the $\tau$-length of a causal curve $\gamma$ as $L(\gamma):=\inf\{\sum_{i=1}^N\tau(\gamma(t_{i-1}),\gamma(t_i))\}$, where the infimum is taken over all partitions $t_0<\dots<t_N$ of the domain of $\gamma$. Under some mild assumptions, this length functional is sequentially upper semicontinuous with respect to uniform convergence, cf.\ \cite[Prop.\ 3.17]{LLS}. If a Lorentzian pre-length space $X$ is locally causally closed, causally path-connected as well as localizable, and its time separation function $\tau$ is given by
$$\tau(x,y)=\sup\{L(\gamma)\,\vert\,\text{$\gamma$ future-directed causal from $x$ to $y$}\},$$
we call $X$ a \textit{Lorentzian length space}. A future-directed causal curve $\gamma$ from $x$ to $y$ with $L(\gamma) = \tau(x,y)$ is called a \textit{maximizer}. Any timelike maximizer $\gamma$ possesses a reparametrization $\tilde\gamma\colon [0,1]\to X$ such that $\tau(\tilde\gamma(0),\tilde\gamma(t))=Ct$ for some constant $C>0$ and all $t\in[0,1]$. Except when the parametrization is explicitly mentioned, we will always tacitly assume timelike maximizers to be parametrized in this way, and whenever we make statements about the uniqueness of maximizers, this is to be understood up to the choice of a reparametrization. Finally, by a \textit{geodesic}, we mean a causal curve which is locally maximizing, i.e., $\gamma\colon [0,1]\to X$ is a geodesic if every $t\in[0,1]$ possesses a neighborhood $I\subseteq[0,1]$ such that $\gamma\big|_I$ is a maximizer.

These concepts allow us to introduce curvature bounds on Lorentzian pre-length spaces. The notion of curvature bound we will be using throughout this article is formulated as a local concavity condition of the time separation:

\begin{defin}\label{def:Concavity}
    A Lorentzian pre-length space $(X,d,\ll,\leq,\tau)$ is called \textit{locally concave} if every point in $X$ possesses a neighborhood $U$ such that:
    \begin{enumerate}[label=(\roman*)]
        \item $\tau\big|_{U\times U}$ is continuous.\label{concav1}
        \item For all $x,y\in U$ with $x\ll y$, there exists a future-directed causal curve $\alpha$ connecting $x$ to $y$ \textit{within} $U$ which satisfies $L(\alpha)=\tau(x,y)$.\label{concav2}
        \item Along any pair of timelike maximizers $\alpha,\beta\colon[0,1]\to U$ with the same time orientation satisfying $\alpha(0)\ll\beta(0)$ or $\alpha(0)=\beta(0)$ and $\alpha(1)\ll\beta(1)$ or $\alpha(1)=\beta(1)$, we have
        \begin{equation*}\label{concav3}
            \tau(\alpha(t),\beta(t))\geq t\tau(\alpha(1),\beta(1))+(1-t)\tau(\alpha(0),\beta(0)) \quad \textnormal{for all } t \in [0,1].
        \end{equation*}
    \end{enumerate}
    Any neighborhood satisfying these conditions will be referred to as a \textit{comparison neighborhood}. 
\end{defin}

\begin{rem}\label{felixprop}
This definition should be compared to other definitions of curvature bounds for Lorentzian pre-length spaces. For a summary of such notions and their interrelations, see \cite{curvature}. Clearly, items \ref{concav1} and \ref{concav2} in \autoref{def:Concavity} are very similar to the respective items in the curvature conditions laid out in \cite{curvature}. Item \ref{concav3}, however, is distinctly different from all other concepts proposed so far.
To shed some light on this condition, we observe that any Lorentzian pre-length space with curvature bounded from above by $0$ in the sense of timelike triangle comparison is locally concave. Indeed, \cite[Prop.\ 6.1]{curvature} ensures that in a space of curvature bounded from above by zero, a curvature comparison neighborhood also constitutes a comparison neighborhood in the sense of Definition \ref{def:Concavity}. Note, however, that this proposition assumes curvature bounds in the sense of \textit{strict causal} triangle comparison, which is not equivalent to timelike triangle comparison in general (cf.\ \cite[Thm.\ 5.1]{curvature}). Nevertheless, timelike triangle comparison turns out to be enough to prove concavity of $\tau$ in our sense as the additional assumption is only needed in the case where the endpoints of the geodesics $\alpha$ and $\beta$ are causally related but not timelike related, a constellation which is not included in \autoref{def:Concavity}. The reason we exclude this case is that one of the principal aims of this article is to show that local concavity globalizes, i.e., that under suitable assumptions, the entire space $X$ can serve as a comparison neighborhood (see \autoref{globalconcav}). The proof of this globalization does not work if the endpoints are merely causally but not timelike related; hence, to remain consistent with what we mean by a comparison neighborhood, we will only impose concavity for maximizers whose endpoints are timelike related or coincide.

\end{rem}

To be able to show global existence and uniqueness of geodesics, we first consider the case in which the two timelike related points are contained in a common comparison neighborhood. 

\begin{prop}\label{uniquemaxlocal}
    Let $X$ be a strongly causal and locally concave regular Lorentzian pre-length space. If $x,y \in X$ are both contained in a comparison neighborhood $U$ and $x\ll y$, then there exists a unique future-directed timelike geodesic from $x$ to $y$ contained entirely in $U$.
\end{prop}

\begin{rem}\label{strongcausalityrem}
    In the subsequent proof, we will rely on strong causality to show that for any point $\gamma(t)$ on a future-directed timelike curve $\gamma$ and any open set $V$ containing it, there exists an $\varepsilon>0$ such that $I(\gamma(t-\varepsilon),\gamma(t+\varepsilon))\subset V$. To see this, we first choose points $x_1,y_1,\dots,x_N,y_N$ in $X$ such that 
    \begin{equation*}
        \gamma(t)\in\Tilde{V}:=\bigcap_{i=1}^NI(x_i,y_i)\subset V.
    \end{equation*}
    Since $\Tilde{V}$ is open, by continuity of $\gamma$, there exists some $\varepsilon>0$ such that $\gamma(t-\varepsilon)$ and $\gamma(t+\varepsilon)$ are both contained in $\Tilde{V}$. This means $x_i\ll \gamma(t-\eps) \ll \gamma(t +\varepsilon)\ll y_i$ for all $1\leq i\leq N$ and hence
    \begin{equation*}
        I(\gamma(t-\varepsilon),\gamma(t+\varepsilon))\subset \Tilde{V}\subset V.
    \end{equation*}
    In particular, timelike diamonds with endpoints on $\gamma$ form a neighborhood basis for $\gamma(t)$. Moreover, if we further assume timelike path-connectedness (meaning that any pair of timelike related points is connected by a timelike curve), we can even conclude that timelike diamonds are a basis of the topology, as, in this setting, any point in $X$ lies on a timelike curve.
\end{rem}

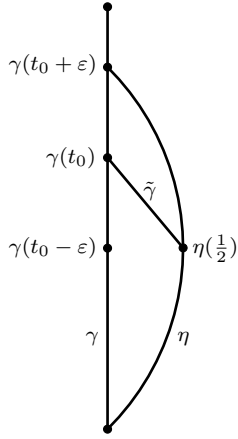
\begin{figure}
\centering
\begin{tikzpicture}[line cap=round,line join=round,>=triangle 45,x=0.4cm,y=0.4cm]
\clip(-3.5,-1.) rectangle (5.,15.);
\draw [line width=1.pt] (0.,0.)-- (0.,14.);
\draw [shift={(-6.,6.)},line width=1.pt]  plot[domain=-0.7853981633974483:0.7853981633974483,variable=\t]({1.*8.485281374238571*cos(\t r)+0.*8.485281374238571*sin(\t r)},{0.*8.485281374238571*cos(\t r)+1.*8.485281374238571*sin(\t r)});
\draw [line width=1.pt] (2.4852761676661537,6.009399916851108)-- (0.,9.);
\begin{scriptsize}
\draw [fill=black] (0.,0.) circle (1.5pt);
\draw [fill=black] (0.,14.) circle (1.5pt);
\draw[color=black] (-0.5,3) node {$\gamma$};
\draw [fill=black] (0.,12.) circle (1.5pt);
\draw[color=black] (-1.8,12.073331483997745) node {$\gamma(t_0+\varepsilon)$};
\draw[color=black] (2.5,3) node {$\eta$};
\draw [fill=black] (2.4852761676661537,6.009399916851108) circle (1.5pt);
\draw[color=black] (3.55,6) node {$\eta(\frac{1}{2})$};
\draw [fill=black] (0.,9.) circle (1.5pt);
\draw[color=black] (-1.2,9) node {$\gamma(t_0)$};
\draw[color=black] (1.4,7.882161783920265) node {$\tilde{\gamma}$};
\draw [fill=black] (0.,6.) circle (1.5pt);
\draw[color=black] (-1.8,6) node {$\gamma(t_0-\varepsilon)$};
\end{scriptsize}
\end{tikzpicture}
\caption{Geodesics in a comparison neighborhood must be maximizers} \label{fig:GeodCompNbhd}
\end{figure}

\begin{proof}[Proof of \autoref{uniquemaxlocal}]
    Existence follows from the definition of a comparison neighborhood. Indeed, \autoref{def:Concavity}\ref{concav2} ensures that $x$ and $y$ can be joined by a causal maximizer within $U$, and this maximizer has to be timelike by regularity (see \cite[Def.\ 3.16]{LLS}). We show uniqueness in two parts:
    \begin{enumerate}[label=(\roman*)]
        \item Every timelike geodesic in $U$ is a timelike maximizer.\label{uniqueness1}
        \item Timelike maximizers in $U$ are uniquely determined by their endpoints.\label{uniqueness2}
    \end{enumerate}
    \ref{uniqueness1} Let $\gamma\colon [0,1] \to U$ be a timelike geodesic in $U$ and suppose that it is not maximizing.  Define $t_0:=\sup \{t\in[0,1]\, \vert \, \gamma\big|_{[0,t]}\text{ is maximizing}\}$. By upper semicontinuity of the length functional and lower semicontinuity of $\tau$, we then have
    \begin{equation*}
        L(\gamma\big|_{[0,t_0]})\geq\limsup_{t\nearrow t_0}L(\gamma\big|_{[0,t]})=\lim_{t\nearrow t_0}\tau(\gamma(0),\gamma(t))\geq\tau(\gamma(0),\gamma(t_0)).
    \end{equation*}
    In particular, $\gamma\big|_{[0,t_0]}$ is maximizing. Now pick $\varepsilon>0$ small enough such that $\gamma$ is maximizing between $\gamma(t_0-\varepsilon)$ and $\gamma(t_0+\varepsilon)$.
    By the definition of $t_0$, we can find a maximizer $\eta\colon [0,1]\to U$ connecting $\gamma(0)$ to $\gamma(t_0 + \eps)$ which is strictly longer than the corresponding segment on $\gamma$, see \autoref{fig:GeodCompNbhd}. We can now use the concavity of the time separation to relate the $\tau$-distance of points along $\eta$ and points on the segment $\gamma\big|_{[t_0-\varepsilon,t_0+\varepsilon]}$:
    \begin{equation*}
        \tau(\eta(\tfrac{1}{2}),\gamma(t_0))\geq\tfrac{1}{2}(\tau(\eta(0),\gamma(t_0-\eps))+\tau(\eta(1),\gamma(t_0+\varepsilon)))= \tfrac{1}{2} \tau(\gamma(0), \gamma(t_0 - \eps)) = \tfrac{1}{2}L(\gamma\big|_{[0,t_0-\varepsilon]})
    \end{equation*}
    Therefore, defining a broken geodesic $\Tilde{\gamma}$ starting at $\gamma(0)$, following $\eta$ until $\eta(\tfrac{1}{2})$ and taking a longest path from there to $\gamma(t_0)$, cf.\ again \autoref{fig:GeodCompNbhd}, we find that
    \begin{equation*}
        L(\Tilde{\gamma})=\tfrac{1}{2}L(\eta)+\tau(\eta(\tfrac{1}{2}),\gamma(t_0))>\tfrac{1}{2}[L(\gamma\big|_{[0,t_0+\varepsilon]})+L(\gamma\big|_{[0,t_0-\varepsilon]})]=L(\gamma\big|_{[0,t_0]}),
    \end{equation*}
    contradicting maximality of $\gamma\big|_{[0,t_0]}$. This finishes the proof of \ref{uniqueness1}.

    \ref{uniqueness2} Suppose there were two distinct length-maximizing curves $\gamma_1,\gamma_2\colon [0,1]\to U$ connecting $x,y \in X$ and let $t \in [0,1]$ be such that $\gamma_1(t) \neq \gamma_2(t)$. By Remark \ref{strongcausalityrem}, there exists some $\eps >0$ with $I(\gamma_1(t-\eps), \gamma_1(t+\eps)) \cap I(\gamma_2(t-\eps), \gamma_2(t+\eps)) = \varnothing$. By comparing the segment joining $\gamma_1(0)$ to $\gamma_1(1-\eps)$ along $\gamma_1$ with that connecting $\gamma_2(\eps)$ to $\gamma_2(1)$ following $\gamma_2$, using concavity of $\tau$, we get
    \begin{equation*}
        \tau(\gamma_1(t), \gamma_2(t+\eps)) \geq \eps \tau(x,y)> 0,
    \end{equation*}
    i.e., $\gamma_1(t) \in I^-(\gamma_2(t+\eps))$. Analogously, we get $\gamma_1(t) \in I^+(\gamma_2(t-\eps))$, a contradiction. \qedhere
    
\end{proof}

\begin{rem}
    In the above proof, we made use of the upper semicontinuity of the length functional, which is proven in \cite[Prop.\ 3.17]{LLS} under the additional assumption of localizability, and which is therefore not directly applicable to our situation. We note, however, that the only part of the definition of localizability (see \cite[Def.\ 3.16]{LLS}) that is needed in the proof of this fact is that every point in $X$ possesses a neighborhood $U$ such that the time separation between two timelike related points in $U$ can be realized by a maximizing geodesic. In our setting, this is always satisfied as it is built into \autoref{def:Concavity}, so we do not need to assume localizability. 
\end{rem}

\begin{rem}
    The same conclusion as in \autoref{uniquemaxlocal}, i.e., that timelike geodesics in comparison neighborhoods are unique, holds true if $X$ satisfies a curvature bound from above by any $\kappa \in\mathbb R$. In fact, under this assumption, the uniqueness of timelike maximizers was proven in \cite[Prop.\ 4.8]{Patchwork}. We now show that any timelike geodesic in a curvature comparison neighborhood must be a maximizer:
    \begin{proof}
        Suppose that $\gamma\colon[0,1]\to U$ is a geodesic from $x$ to $y$, which does not maximize the $\tau$-distance between its endpoints. Then, without loss of generality, we can assume the existence of some $c\in[0,1]$ such that $\gamma\big|_{[0,c]}$ and $\gamma\big|_{[c,1]}$ are maximal (this can always be achieved by arguing as in \ref{uniqueness1} of the previous proof). Setting $z:=\gamma(c)$, we know that there exist nearby points $p\ll z\ll q$ on the image of $\gamma$ such that $\gamma$ realizes the distance between $p$ and $q$. We now consider a comparison triangle $\Delta \Bar{x}\Bar{y}\Bar{z}$ of $\Delta xyz$ in the appropriate comparison space $M_\kappa$ and choose corresponding points $\Bar{p},\Bar{q}$ on $[\Bar{x},\Bar{z}]$ and $[\Bar{z},\Bar{y}]$ respectively. Then, the curvature bound implies
        \begin{equation*}
            \tau(p,q)\geq\Bar{\tau}(\Bar{p},\Bar{q})>\Bar{\tau}(\Bar{p},\Bar{z})+\Bar{\tau}(\Bar{z},\Bar{q})=\tau(p,z)+\tau(z,q).
        \end{equation*}
        However, this is a contradiction to $\gamma$ being a maximizer between $p$ and $q$. 
    \end{proof}
\end{rem}

\begin{cor}\label{geodesiclimit}
    Let $X$ be a strongly causal regular Lorentzian pre-length space which is either locally concave or has curvature bounded from above by some $\kappa \in\mathbb R$. Let $(\gamma_k)_{k\in\mathbb N}$ be a sequence of timelike geodesics converging uniformly to a timelike curve $\gamma$. Then $\gamma$ is a geodesic. 
\end{cor}

\begin{proof}
    Fix $t \in [0,1]$ and let $U$ be a comparison neighborhood of $\gamma(t)$. Choose $\varepsilon>0$ such that $\gamma\big|_{(t-\varepsilon,t+\varepsilon)}$ remains in $U$. As $(\gamma_k)_k$ converges uniformly to $\gamma$ on $(t-\eps, t+\eps)$, we can assume, without loss of generality, that $\gamma_k\big|_{(t-\varepsilon,t+\varepsilon)}$ remains in $U$ for all $k$. Combining the previous proposition and remark, we conclude that $\gamma_k\big|_{(t-\varepsilon,t+\varepsilon)}$ is a maximizer for each $k$. Now, by upper semicontinuity of the length functional and lower semicontinuity of $\tau$, the limit of maximizers in a strongly causal Lorentzian pre-length space is again a maximizer. In particular, $\gamma$ is a maximizer in a neighborhood of any fixed $t \in [0,1]$ and is therefore a geodesic. 
\end{proof}

\section{Local Existence and Uniqueness}\label{sec:LocalExistence}

The crucial step towards proving global existence and uniqueness of geodesics between timelike related points is to first do it in a neighborhood of a geodesic. This can be understood as a generalization of the fact that timelike geodesics in Lorentzian manifolds with sectional curvature bounded from above by $0$ do not encounter conjugate points \cite[10.20.\ Cor.]{oneil}. To be more precise, we want to prove the following:

\begin{thm}\label{großesThm}
    Let $X$ be a globally hyperbolic, strongly causal, locally causally closed, timelike path-connected regular Lorentzian pre-length space and let $\alpha\colon[0,L]\to X$ be a timelike geodesic. If $X$ is locally concave then there exists an open set $U$ and two causally convex open sets $U_0,U_L\subseteq U$ with $\alpha(0)\in U_0,\,\alpha(L)\in U_L$ and $\alpha([0,L])\subseteq U$ such that:
    \begin{enumerate}[label=(\alph*)]
        \item\label{großesThma} For any pair $(p,q)\in U_0 \times U_L$ there exists a unique timelike geodesic $\gamma\colon [0,L]\to X$ with $\gamma(0)=p,\, \gamma(L)=q$ and $\gamma([0,L])\subset U$. 
        \item\label{großesThmb} Within the set $U$, $\gamma$ is the (unique) longest curve from $p$ to $q$. In particular, if $p\leq \alpha(0)$ and $\alpha(L)\leq q$, then
        \begin{equation}\label{Längenungleichung}
            L(\gamma)\geq L(\alpha)+\tau(p,\alpha(0))+\tau(\alpha(L), q).
        \end{equation}
        \item\label{großesThmc} Whenever $p_k\to\alpha(0)$ and $q_k\to\alpha(L)$ for sequences $(p_k)_{k \in \N}\subseteq U_0$ and $(q_k)_{k \in \N}\subseteq U_L$, the corresponding curves $\gamma_k$ connecting $p_k$ to $q_k$ converge to $\alpha$ uniformly.
        \item \label{großesThmd} If $\alpha(0)\ll p$ or $\alpha(0)=p$ and $\alpha(L)\ll q$ or $\alpha(L)=q$, then the function $t\mapsto\tau(\alpha(t),\gamma(t))$ is concave. The analogous statement holds if $\alpha(0)\gg p$ and $\alpha(L)\gg q$.
    \end{enumerate}
\end{thm}

To prove the existence of an appropriate neighborhood of $\alpha([0,L])$, we make use of the following fact:

\begin{lem}\label{deltaexists}
    With the same setup as in \autoref{großesThm}, we can find some $\delta>0$ such that the sets of the form $I(\alpha(t),\alpha(t+\delta))$ are comparison neighborhoods for all $t\in[0, L-\delta]$.
\end{lem}

\begin{proof}
    Both $\alpha(0)$ and $\alpha(L)$ are contained in comparison neighborhoods $U_0$ and $U_L$, respectively. As in \autoref{strongcausalityrem}, there exists an $\varepsilon_->0$ such that $I(\alpha(0),\alpha(\varepsilon_-))\subset U_0$ and so, in particular, $I(\alpha(0),\alpha(\varepsilon_-))$ is a comparison neighborhood. Similarly, we can find an $\varepsilon_+ > 0$ such that $I(\alpha(\varepsilon_+),\alpha(L))$ is a comparison neighborhood. Set $\varepsilon_1:=\min\{\varepsilon_-,\varepsilon_+\}$.

    Now cover $\alpha\big|_{[\varepsilon_1/2, L-\varepsilon_1/2]}$ with finitely many timelike comparison diamonds $I(\alpha(t_0),\alpha(t_0+\delta_0)),\ldots, I(\alpha(t_n),\alpha(t_n+\delta_n))$, where $t_0<t_1<t_0+\delta_0<t_2<t_1+\delta_1<\dots<t_n+\delta_n$ and define $\varepsilon_2:=\min\{(t_i+\delta_i)-t_{i+1}\,|\,0\leq i\leq n-1\}$. Setting $\delta:=\min\{\varepsilon_1,\varepsilon_2\}$ then gives the claim. 
\end{proof}

Before explicitly defining the sets $U_0, U_L$ and $U$, we introduce some shorthand notation:
\begin{align*}
    &D(\alpha(t),\varepsilon):=I(\alpha(t-\tfrac{\varepsilon}{2}),\alpha(t+\tfrac{\varepsilon}{2})),\\
    &\overline{D}(\alpha(t),\varepsilon):=J(\alpha(t-\tfrac{\varepsilon}{2}),\alpha(t+\tfrac{\varepsilon}{2}))\\
    &D(\alpha,\varepsilon) := \mkern-28mu \smashoperator[r]{\bigcup_{
\substack{t\in[\sfrac{\eps}{2},L-\sfrac{\eps}{2}]}} }
\, I(\alpha(t-\tfrac{\varepsilon}{2}),\alpha(t+\tfrac{\varepsilon}{2})),
\end{align*}
so $D(\alpha(t),\varepsilon)$ describes a timelike $\eps$-diamond centered at $\alpha(t)$ and $D(\alpha,\varepsilon)$ is the union of all such diamonds, covering the entire curve $\alpha$ except for its endpoints. We now choose an $\varepsilon>0$ satisfying the following properties:
\begin{enumerate}[label=(\roman*)]
    \item\label{epsilon2} There exists some $\delta>0$ so that $D(\alpha(t),3\varepsilon+\delta)$ is comparison neighborhood for any $t$. 
    \item\label{epsilon1} We can find points $x\ll\alpha(0)$ and $y\gg\alpha(L)$ such that $I(x,\alpha(\tfrac{5}{2}\varepsilon))$ and $I(\alpha(L-\tfrac{5}{2}\varepsilon),y)$ are comparison neighborhoods. 
    \item\label{epsilon3} There exists a natural number $N$ such that $L=N\varepsilon$.
\end{enumerate}

\begin{rem}\label{etadefinition}
To see that such an $\eps>0$ can actually be found, first notice that \ref{epsilon2} is a direct consequence of \autoref{deltaexists} and \autoref{epsilon3} is just a matter of convenience. We now claim that \ref{epsilon1} is implied by the assumptions of strong causality and timelike path-connectedness. Indeed, the definition of strong causality guarantees the existence of some $z\ll\alpha(0)$ and by timelike path-connectedness, there exists a timelike curve $\sigma\colon [0,1]\to X$ from $z$ to $\alpha(0)$. So if $U$ is any causally convex comparison neighborhood of $\alpha(0)$, then for large enough $t$, we will have $\sigma(t)\in U$ and can therefore choose $x=\sigma(t)$. The point $y$ can be constructed similarly. 

The reason for demanding \autoref{epsilon1} is the inconvenient circumstance that the sets of the form $D(\alpha(t),\varepsilon)$ do not cover the endpoints of $\alpha$. This issue could, of course, easily be dealt with if we knew that the unique maximal geodesics $\beta_1\colon [0,\frac{\varepsilon}{2}]\to X$ from $x$ to $\alpha(0)$ and $\beta_2 \colon [0, \frac{\eps}{2}] \rightarrow X$ from $\alpha(L)$ to $y$ were parts of a geodesic prolongation of $\alpha$, but this need not be true in general. However, for our purposes, it will suffice to choose $x \ll \alpha(0)$ (and similarly $y$) in such a way that the maximal geodesics $\eta\colon [0,\frac{3}{2}\varepsilon]\to X$ from $x$ to $\alpha(\frac{3}{2}\varepsilon)$ and $\Tilde{\eta}\colon [0,\frac{\varepsilon}{2}]\to X$ from $x$ to $\alpha(\frac{\varepsilon}{2})$ are \textit{sufficiently close} to $\alpha$. To be more precise, we want to achieve 
\begin{equation*}
\eta\left(\tfrac{3}{4}\varepsilon\right)\in D\left(\alpha\left(\tfrac{3}{4}\varepsilon\right),\tfrac{\varepsilon}{2}\right)=I\left(\alpha\left(\tfrac{\varepsilon}{2}\right),\alpha(\varepsilon)\right)
\end{equation*}
and 
\begin{equation*}
\Tilde{\eta}\left(\tfrac{\varepsilon}{4}\right)\in D\left(\alpha\left(\tfrac{\varepsilon}{4}\right),\tfrac{\varepsilon}{2}\right)=I\left(\alpha(0),\alpha\left(\tfrac{\varepsilon}{2}\right)\right),
\end{equation*}
cf.\ \autoref{fig:ConstrEta}.

\begin{figure}
\centering
\begin{tikzpicture}[line cap=round,line join=round,>=triangle 45,x=1.0cm,y=1.0cm]
\clip(-4.5,-4.5) rectangle (0.5,5.5);
\fill[line width=1.pt,fill=black,fill opacity=0.10000000149011612] (-2.,0.) -- (-1.,1.) -- (-2.,2.) -- (-3.,1.) -- cycle;
\fill[line width=2.pt,fill=black,fill opacity=0.10000000149011612] (-2.,0.) -- (-1.,-1.) -- (-2.,-2.) -- (-3.,-1.) -- cycle;
\draw [line width=1.pt] (-2.,-2.)-- (-1.,-4.);
\draw [line width=1.pt] (-2.,-2.)-- (-1.,-1.);
\draw [line width=1.pt] (-2.,0.)-- (-1.,-1.);
\draw [line width=1.pt] (-2.,-2.)-- (-3.,-1.);
\draw [line width=1.pt] (-3.,-1.)-- (-2.,0.);
\draw [line width=1.pt] (-2.,0.)-- (0.,2.);
\draw [line width=1.pt] (0.,2.)-- (-2.,4.);
\draw [line width=1.pt] (-2.,4.)-- (-4.,2.);
\draw [line width=1.pt] (-4.,2.)-- (-2.,0.);
\draw [line width=1.pt] (-2.,2.)-- (-1.,1.);
\draw [line width=1.pt] (-2.,2.)-- (-3.,1.);
\draw [line width=1.pt] (-2.,4.)-- (-1.,5.);
\draw [line width=1.pt] (-2.,4.)-- (-3.,5.);
\draw [line width=1.pt] (-1.5,-3.)-- (-2.,0.);
\draw [line width=1.pt] (-1.5,-3.)-- (-2.,4.);
\draw [line width=1.pt] (-1.5,-3.)-- (-0.25,-1.75);
\draw [line width=1.pt] (-0.25,-1.75)-- (-2.,0.);
\draw [line width=1.pt] (-1.5,-3.)-- (-3.25,-1.25);
\draw [line width=1.pt] (-3.25,-1.25)-- (-2.,0.);
\draw [line width=1.pt] (-2.,-2.)-- (-2.,5.);
\draw [line width=1.pt] (-2.,0.)-- (-1.,1.);
\draw [line width=1.pt] (-1.,1.)-- (-2.,2.);
\draw [line width=1.pt] (-2.,2.)-- (-3.,1.);
\draw [line width=1.pt] (-3.,1.)-- (-2.,0.);
\draw [line width=1.pt] (-2.,0.)-- (-1.,-1.);
\draw [line width=1.pt] (-1.,-1.)-- (-2.,-2.);
\draw [line width=1.pt] (-2.,-2.)-- (-3.,-1.);
\draw [line width=1.pt] (-3.,-1.)-- (-2.,0.);
\begin{scriptsize}
\draw [fill=black] (-2.,-2.) circle (2.0pt);
\draw [fill=black] (-2.,0.) circle (2.0pt);
\draw[color=black] (-2.642179432019956,0.06141986190377435) node {$\alpha(\varepsilon/2)$};
\draw [fill=black] (-2.,4.) circle (2.0pt);
\draw[color=black] (-2.8,4.03) node {$\alpha(3\varepsilon/2)$};
\draw [fill=black] (-1.,-4.) circle (2.0pt);
\draw[color=black] (-0.7543409716180294,-4) node {$x$};
\draw[color=black] (-1.4051858138233413,-3.6498989295681756) node {$\beta_1$};
\draw [fill=black] (-1.5,-3.) circle (2.0pt);
\draw[color=black] (-1.000658716050395,-3.070675765581224) node {$\beta_1(t)$};
\draw [fill=black] (-1.75,-1.5) circle (2.0pt);
\draw[color=black] (-1.05193749747141628,-1.6977023398343751) node {$\tilde{\eta}_t(\varepsilon/4)$};
\draw [fill=black] (-1.75,0.5) circle (2.0pt);
\draw[color=black] (-1.0197129115962874,0.3555080605457944) node {$\eta_t(3\varepsilon/4)$};
\draw[color=black] (-2.2525678500264785,2.9039039073952964) node {$\alpha$};
\end{scriptsize}
\end{tikzpicture}
\caption{Construction of $\eta$ and $\tilde{\eta}$} \label{fig:ConstrEta}
\end{figure}
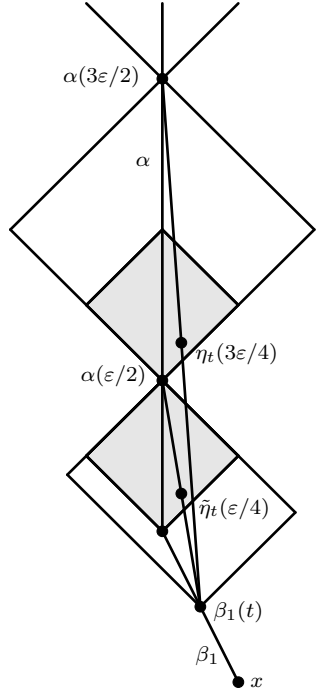

This is not immediate, but we can always find some $t\in[0,1]$ for which the maximal geodesics $\eta_t$ from $\beta_1(t)$ to $\alpha(\frac{3}{2}\varepsilon)$ and $\Tilde{\eta}_t$ from $\beta_1(t)$ to $\alpha(\frac{\varepsilon}{2})$ satisfy these conditions. Indeed, this is a consequence of the limit curve theorem \cite[Thm. 3.7]{LLS} which (up to subsequences) ensures uniform convergence of $\eta_t$ and $\Tilde{\eta}_t$ to some limit curves as $t$ approaches $\frac{\varepsilon}{2}$ and these limit curves must then coincide with the respective restrictions of $\alpha$ to $[0,3\varepsilon/2]$ and $[0,\varepsilon/2]$ by upper semicontinuity of the length functional and lower semicontinuity of $\tau$. This allows us to choose a parameter $t \in [0,1]$ as claimed, and we will assume, without loss of generality, that $t=0$. For the endpoint $\alpha(L)$, we similarly find maximal curves $\zeta$ and $\Tilde{\zeta}$ from $\alpha(L-\frac{3}{2}\varepsilon)$ to $y$, respectively from $\alpha(L-\frac{\varepsilon}{2})$ to $y$ which are sufficiently close to $\alpha$ by applying the same argument to $y \gg \alpha(L)$ and $\beta_2$.
Finally, we define $\Tilde{\alpha}\colon [-\frac{\varepsilon}{2},L+\frac{\varepsilon}{2}]\to X$ as the concatenation $\Tilde{\alpha}:=\beta_1*\alpha*\beta_2$. 
\end{rem}

We can now introduce appropriate neighborhoods of $\alpha(0),\alpha(L)$ and $\alpha([0,L])$ as follows:
\begin{itemize}
    \item[--] $U_0\, :=D(\Tilde{\alpha}(0),\varepsilon)$
    \item[--] $U_L:=D(\Tilde{\alpha}(L),\varepsilon)$
    \item[--] $U\;\, :=D(\Tilde{\alpha},\varepsilon)$
\end{itemize}

Note that our choice of $\eps > 0$ is precisely tailored to ensure that three disjoint comparison neighborhoods of the form $D(\Tilde{\alpha}(t),\varepsilon)$ fit into a common larger comparison neighborhood (cf.\ conditions \ref{epsilon2} and \ref{epsilon1}). Moreover, exactly $N$ such neighborhoods centered at the points $p_0:=\alpha(0),p_1:=\alpha(\varepsilon),\ldots,p_N:=\alpha(N\varepsilon)=\alpha(L)$ fit into the set $U$ by condition \ref{epsilon3}. The union of these sets $D(p_i,\varepsilon)$ covers all of $\Tilde{\alpha}$ except for the points 
$$x_0:=\Tilde{\alpha}(-\tfrac{\varepsilon}{2}),\,x_1:=\Tilde{\alpha}(\tfrac{\varepsilon}{2}),\,x_2:=\Tilde{\alpha}(\tfrac{3\varepsilon}{2}), \, \dots\, ,x_{N+1}:=\Tilde{\alpha}(L+\tfrac{\varepsilon}{2}).$$

\begin{proof}[Proof of \autoref{großesThm}]
    We can assume, without loss of generality, that $\alpha\colon [0,L]\to X$ is parametrized by arc length.

    \textbf{Existence:} Let $p\in D(p_0,\varepsilon), q\in D(p_N,\varepsilon)$ and consider tuples $(a_0,\ldots,a_N)$ satisfying $a_0=p$, $a_N=q$ and $a_i\in\overline{D}(p_i,\varepsilon)$ for all $0 \leq i \leq N$, see \autoref{fig:SetupGroßesThm}. Such tuples will be termed \textit{admissible}. The set $T$ of all admissible tuples is compact, since any causal diamond $\overline{D}(p_i,\varepsilon)$ is compact by the assumption of global hyperbolicity. As a result, there must exist a tuple that maximizes the sum
    \begin{equation}\label{tausum}
        \sum_{i=0}^N\tau(a_i,a_{i+1}).
    \end{equation}
    We want to construct a broken timelike geodesic from $p$ to $q$ by successively connecting the points $a_i$. For this to work, we need to ensure that $a_i\in D(p_i,\varepsilon)=I(x_{i},x_{i+1})$, i.e., exclude the possibility $a_i\in\overline{D}(p_i,\varepsilon)\setminus D(p_i,\varepsilon)$ as this might introduce a null segment. To do so, we argue that we can always replace $a_i$ by some $\bar a_i$ contained in $D(p_i,\varepsilon)$ without decreasing the value of the sum in \eqref{tausum}.

    First, note that $a_0\in I(x_0,x_1)$ by assumption. Next, we try to find an appropriate $\bar{a}_1\in I(x_1,x_2)$. As the points $a_{0}$ and $a_{2}$ are both contained in a common comparison neighborhood, by \autoref{uniquemaxlocal}, there exists a unique maximal geodesic $[a_{0},a_{2}]\colon [0,1]\to X$ joining them. 
    Henceforth, we will use this notation to denote the unique maximal segment between two points in a comparison neighborhood. We now claim that $\bar{a}_1:=[a_0,a_2](\frac{1}{2})$ is a suitable choice. In fact, the maximality of $[a_0,a_2]$ guarantees that replacing $a_1$ by $\bar{a}_1$ does not decrease the value of the sum in \eqref{tausum}, and so it remains to show that $\bar{a}_1\in I(x_1,x_2)$. We start by arguing that $\bar{a}_1\in I^-(x_2)$. To do so, consider the maximal segment $[x_1,x_3]\colon[0,1]\to X$ and note that $a_0\ll x_1$ as well as $a_2\leq x_3$.
    If we had $a_2\ll x_3$, we could directly use local concavity to compare the midpoints of the segments $[a_0,a_2]$ and $[x_1,x_3]$. 

\begin{minipage}{0.35\textwidth}
\centering
\begin{tikzpicture}[line cap=round,line join=round,>=triangle 45,x=1.3cm,y=1.3cm]
x=1.0cm,y=1.0cm,
axis lines=middle,
ymajorgrids=true,
xmajorgrids=true,
xmin=2.0,
xmax=6,
ymin=-4.1,
ymax=6.1,
xtick={2.0,4.0,...,6.0},
ytick={-4.0,-2.0,...,6.0},]
\clip(2.,-4.1) rectangle (6.,6.1);
\fill[line width=1.pt,fill=black,fill opacity=0.10000000149011612] (4.,2.) -- (3.,1.) -- (4.,0.) -- (5.,1.) -- cycle;
\fill[line width=1.pt,fill=black,fill opacity=0.10000000149011612] (4.,4.) -- (5.,5.) -- (4.138752428286914,5.861247571713086) -- (3.138752428286914,4.861247571713086) -- cycle;
\fill[line width=1.pt,fill=black,fill opacity=0.10000000149011612] (4.,0.) -- (3.,-1.) -- (4.,-2.) -- (5.,-1.) -- cycle;
\fill[line width=1.pt,fill=black,fill opacity=0.10000000149011612] (4.155046219775697,-3.8449537802243032) -- (5.,-3.) -- (4.,-2.) -- (3.1550462197756968,-2.8449537802243032) -- cycle;
\draw [line width=1.pt] (4.,2.)-- (3.,1.);
\draw [line width=1.pt] (3.,1.)-- (4.,0.);
\draw [line width=1.pt] (4.,0.)-- (5.,1.);
\draw [line width=1.pt] (5.,1.)-- (4.,2.);
\draw [line width=1.pt] (4.,4.)-- (5.,5.);
\draw [line width=1.pt] (5.,5.)-- (4.138752428286914,5.861247571713086);
\draw [line width=1.pt] (4.138752428286914,5.861247571713086)-- (3.138752428286914,4.861247571713086);
\draw [line width=1.pt] (3.138752428286914,4.861247571713086)-- (4.,4.);
\draw [line width=1.pt] (4.,0.)-- (3.,-1.);
\draw [line width=1.pt] (3.,-1.)-- (4.,-2.);
\draw [line width=1.pt] (4.,-2.)-- (5.,-1.);
\draw [line width=1.pt] (5.,-1.)-- (4.,0.);
\draw [line width=1.pt,] (4.,5.)-- (4.138752428286914,5.861247571713086);
\draw [line width=1.pt,] (4.,5.)-- (4.,3.5);
\draw [line width=1.pt,dotted] (4.,3.5)-- (4.,2.5);
\draw [line width=1.pt] (4.155046219775697,-3.8449537802243032)-- (5.,-3.);
\draw [line width=1.pt] (5.,-3.)-- (4.,-2.);
\draw [line width=1.pt] (4.,-2.)-- (3.1550462197756968,-2.8449537802243032);
\draw [line width=1.pt] (3.1550462197756968,-2.8449537802243032)-- (4.155046219775697,-3.8449537802243032);
\draw [line width=1.pt] (4.155046219775697,-3.8449537802243032)-- (4.,-3.);
\draw [line width=1.pt] (4.,-3.)-- (4.,2.5);
\draw [line width=1.pt] (5.,-3.)-- (5.,2.5);
\draw [line width=1.pt] (3.1550462197756968,-2.8449537802243032)-- (3.,-2.);
\draw [line width=1.pt] (3.,-2.)-- (3.,2.5);
\draw [line width=1.pt] (5.,5.)-- (5.,3.5);
\draw [line width=1.pt] (3.138752428286915,4.861247571713086)-- (3.,4.);
\draw [line width=1.pt] (3.,4.)-- (3.,3.5);
\draw [line width=1.pt,dotted] (3.,3.5)-- (3.,2.5);
\draw [line width=1.pt,dotted] (5.,3.5)-- (5.,2.5);
\draw [shift={(35.91604905004076,0.968371488097471)},line width=1.pt]  plot[domain=3.0128640951891827:3.0611824841302444,variable=\t]({1.*31.495415833676777*cos(\t r)+0.*31.495415833676777*sin(\t r)},{0.*31.495415833676777*cos(\t r)+1.*31.495415833676777*sin(\t r)});
\draw [shift={(35.91604905004076,0.968371488097471)},line width=1.pt,dotted]  plot[domain=3.0611824841302444:3.0928781654044126,variable=\t]({1.*31.495415833676777*cos(\t r)+0.*31.495415833676777*sin(\t r)},{0.*31.495415833676777*cos(\t r)+1.*31.495415833676777*sin(\t r)});
\draw [shift={(35.91604905004076,0.968371488097471)},line width=1.pt]  plot[domain=3.0928781654044126:3.2691057325879505,variable=\t]({1.*31.49541583367678*cos(\t r)+0.*31.49541583367678*sin(\t r)},{0.*31.49541583367678*cos(\t r)+1.*31.49541583367678*sin(\t r)});
\begin{scriptsize}
\draw [fill=black] (4.,2.) circle (2.pt);
\draw[color=black] (3.75,2) node {$x_3$};
\draw [fill=black] (4.,0.) circle (2.pt);
\draw[color=black] (3.75,0) node {$x_2$};
\draw [fill=black] (4.,4.) circle (2.pt);
\draw[color=black] (3.75,4) node {$x_N$};
\draw [fill=black] (4.138752428286914,5.861247571713086) circle (2.pt);
\draw[color=black] (3.75,5.86) node {$x_{N+1}$};
\draw [fill=black] (4.,-2.) circle (2.pt);
\draw[color=black] (3.75,-2) node {$x_1$};
\draw [fill=black] (4.,5.) circle (2.pt);
\draw[color=black] (3.75,5) node {$p_N$};
\draw [fill=black] (4.155046219775697,-3.8449537802243032) circle (2.pt);
\draw[color=black] (3.9,-3.84) node {$x_0$};
\draw [fill=black] (4.,-3.) circle (2.pt);
\draw[color=black] (3.75,-3) node {$p_0$};
\draw [fill=black] (4.,-1.) circle (2.pt);
\draw[color=black] (3.75,-1) node {$p_1$};
\draw [fill=black] (4.,1.) circle (2.pt);
\draw[color=black] (3.75,1) node {$p_2$};
\draw [fill=black] (4.681228981385298,5.0115427452836085) circle (2.pt);
\draw[color=black] (4.45,5.01) node {$a_N$};
\draw [fill=black] (4.67,-3.0424279142553354) circle (2.pt);
\draw[color=black] (4.45,-3.04) node {$a_0$};
\draw[color=black] (4.6,0) node {$\gamma$};
\draw[color=black] (3.85,-0.5) node {$\alpha$};
\draw [fill=black] (4.481098048089429,-0.9822884107923868) circle (2.pt);
\draw[color=black] (4.717586474037168,-0.98) node {$a_1$};
\draw [fill=black] (4.420667854671432,1.0150822426837798) circle (2.pt);
\draw[color=black] (4.64477776662083,1.01) node {$a_2$};
\end{scriptsize}
\end{tikzpicture}
\captionof{figure}{Setup} \label{fig:SetupGroßesThm}

\end{minipage}
\hfill
\begin{minipage}{0.6\textwidth}
    In the general case, we introduce the points $a_2^k:=[a_0,a_2](1-\frac{1}{k})$ and $\bar{a}_1^k:=[a_0,a_2^k](\frac{1}{2})$. Then, $a_2^k\to a_2$ as well as $\bar{a}_1^k\to \bar{a}_1$ for $k\to\infty$ and by push-up (see \cite[Lem.\ 2.10]{LLS}) $a_2^k\ll x_3$ for all $k\in\mathbb N$.
    Hence, we can use the concavity of the time separation between the segments $[a_0,a_2^k]$ and $[x_1,x_3]$ to infer
    \begin{equation*}
    \begin{split}
        \tau(\bar{a}_1,x_2)&=\lim_{k\to\infty}\tau(\bar{a}_1^k,x_2)\geq\\
        &\geq\lim_{k\to\infty}\tfrac{1}{2}\left(\tau(a_0,x_1)+\tau(a_2^k,x_3)\right)=\\
        &=\tfrac{1}{2}\left(\tau(a_0,x_1)+\tau(a_2,x_3)\right)>0
    \end{split}
    \end{equation*}
    i.e., $\bar{a}_1\in I^-(x_2)$. Note that in the above inequality, we used that the time separation restricted to comparison neighborhoods is continuous. To show $\bar{a}_1\in I^+(x_1)$, we compare $[a_0,a_2]$ to the curve $\eta$ defined in \autoref{etadefinition}, i.e., the maximal geodesic from $x_0$ to $x_2$ which satisfies $\eta(\frac{1}{2})\gg x_1$. Then, as above, the concavity of $\tau$ grants
    \begin{equation*}
        \tau(x_1,\bar{a}_1)\geq\tau(\eta(\tfrac{1}{2}),\bar{a}_1)\geq\frac{1}{2}(\tau(x_0,a_0)+\tau(x_2,a_2))>0,
    \end{equation*}
    in particular, $\bar{a}_1\in I(x_1,x_2)$. Continuing to argue like this (without having to work with $\eta$), we can show inductively that all $a_i$ can be assumed to lie in $I(x_i,x_{i+1})=D(p_i,\varepsilon)$ for $i\leq N-1$. For $i=N$, this once again holds by assumption.\\

    As a result of this, we know that the concatenation
    \begin{equation*}
        \gamma:=[a_0,a_1]*[a_1,a_2]*\dots*[a_{N-1},a_N]
    \end{equation*}
\end{minipage}\\

    defines a broken timelike geodesic connecting $p$ to $q$. We claim, moreover, that (upon reparametrization) this curve is, in fact, a geodesic, i.e., that $\gamma$ remains a local maximizer in a neighborhood of any $a_i$.
    For this, we argue as follows:
    
    Define $\bar a_i:=[a_{i-1},a_{i+1}](\frac{1}{2})$ and apply concavity as before to conclude $\bar a_i \in D(p_i,\varepsilon)$. In particular, $(a_0,\ldots,\bar a_i,\ldots,a_N)$ is an admissible tuple and by maximality of $[a_{i-1},a_{i+1}]$ we have 
    \begin{equation*}
        \tau(a_0,a_1)+\cdots+\tau(a_{i-1},\bar a_i)+\tau(\bar a_i,a_{i+1})+\cdots\tau(a_{N-1},a_N)\geq\sum_{i=0}^N\tau(a_i,a_{i+1}).
    \end{equation*}
    But since $(a_0,\ldots,a_i,\ldots,a_N)$ already maximizes \eqref{tausum}, this can only be true if equality holds and, in particular, $\tau(a_{i-1},\bar a_i)+\tau(\bar a_i,a_{i+1})=\tau(a_{i-1},a_i)+\tau(a_i,a_{i+1})$. Altogether, we have
    \begin{equation*}
    \begin{split}
        L([a_{i-1},a_i]*[a_i,a_{i+1}])&=\tau(a_{i-1},a_i)+\tau(a_i,a_{i+1})=\\
        &=\tau(a_{i-1},\bar a_i)+\tau(\bar a_i,a_{i+1})=L([a_{i-1},a_{i+1}]),
    \end{split}
    \end{equation*}
    so $[a_{i-1},a_i]*[a_i,a_{i+1}]$ is maximizing and, hence, $\gamma$ is a geodesic.
    
    To finish the existence part of the proof, it remains to show that $\gamma$ does not leave the set $U$. 
    
    \textbf{Claim:} The geodesic $\gamma\colon [0,L]\to X$ satisfies
    \begin{equation}\label{claim}
        \gamma(t)\in D(\Tilde{\alpha}(t),\varepsilon) \quad \text{for all } t \in [0,L].
    \end{equation}
    \textit{Proof of claim.} We first show $\gamma(i\varepsilon)\in D(p_i,\varepsilon)$ for $0 \leq i \leq N$. Note that for the points $a_i$ we already have $a_i\in D(p_i,\varepsilon)$ but they would only satisfy $a_i=\gamma(i\varepsilon)$ if they were equally spaced, i.e., if
    \begin{equation}\label{equallyspaced}
        \tau(a_0,a_1)=\tau(a_1,a_2)=\dots=\tau(a_{N-1},a_N).
    \end{equation}
    As this will not necessarily be the case, we successively replace the points $a_i$ using the following scheme:
    \begin{align*}
        &a_i^0:=a_i,\hspace{0.5cm}&&\text{ for all } 0\leq i\leq N,\\
        &a_0^k:=a_0,\,a_N^k:=a_N,\,a_i^k:=[a_{i-1}^{k-1},a_{i+1}^{k-1}](\tfrac{1}{2}),\hspace{0.5cm}&&\text{ for all }1\leq i\leq N-1, \, k \geq 1.
    \end{align*}
    By the same argument as earlier in the proof, we have $a_i^k\in D(p_i,\varepsilon)$ with $a_i^k$ contained in the image of $\gamma$. In particular, we can assume w.l.o.g.\ that $a_i^k \rightarrow \tilde{a}_i\in \overline{D}(p_i,\varepsilon)$ by compactness of the causal diamonds. In fact, we can even show that $\tilde{a}_i\in D(p_i,\varepsilon)$ (using that $\tilde{a}_0=a_0\in D(p_0,\varepsilon)$ and arguing as above). Finally, the fact that the points $\tilde{a}_i$ satisfy \eqref{equallyspaced} will be proven in the appendix (see \autoref{problem}).

    We conclude that $\tilde{a}_i=\gamma(i\varepsilon)\in D(p_i,\varepsilon)$ as desired. For $t\in\left((i-1)\varepsilon,i\varepsilon\right)$, where $2\leq i\leq N-1$, the inclusion \eqref{claim} is again a consequence of the concavity of $\tau$, comparing points on $[\tilde{a}_{i-1},\tilde{a}_i]$ to points on $[x_{i-1},x_i]$ and $[x_i,x_{i+1}]$, respectively. This is the same argument we have used repeatedly throughout this proof. As before, the cases $t\in (0,\varepsilon)$ and $t\in((N-1)\varepsilon,L)$ require special care, and we give a short argument treating the former case, with the latter being analogous. 
    We start by checking $\gamma(\frac{\varepsilon}{2})\in I(\alpha(0),\alpha(\varepsilon))$. Showing $\gamma(\frac{\varepsilon}{2})\ll\alpha(\varepsilon)$ works as always. Showing $\gamma(\varepsilon)\gg\alpha(0)$ works by proving $\gamma(\frac{\varepsilon}{2})\gg\Tilde{\eta}(\frac{\varepsilon}{2})$ (as always by using concavity) and using that $\Tilde{\eta}(\frac{\varepsilon}{2})\gg\alpha(0)$. Recall that $\Tilde{\eta}$ was defined in \autoref{etadefinition} as the maximal geodesic connecting $x_0$ to $x_1$. For all remaining $t$ we can once again prove $\gamma(t)\in D(\Tilde{\alpha}(t),\varepsilon)$ with the same argument, however to show $\gamma(t)\gg\Tilde{\alpha}(t)$ for $t<\frac{\varepsilon}{2}$ one needs to use the maximal geodesic from $x_0$ to $p_0$.

    This finishes the proof of the claim and, thereby, also of the existence part of \autoref{großesThm}. Before we continue, observe that \eqref{claim}, in fact, holds for \textit{any} timelike geodesic satisfying the assumptions of \autoref{großesThm}. Indeed, for such a curve $\gamma$, we can choose points $a_i\in D(p_i,\varepsilon)$ on the image of $\gamma$ and proceed as above.

    \textbf{Uniqueness:} By way of contradiction, suppose we could find two distinct timelike geodesics $\gamma_1,\gamma_2\colon [0,L]\to X$ satisfying the assumptions of \autoref{großesThm}. Define the function
    \begin{equation*}
        f\colon [0,L-\varepsilon]\to\mathbb{R},\hspace{1cm}t\mapsto\tau(\gamma_1(t),\gamma_2(t+\varepsilon)).
    \end{equation*}
    As explained above, we can now apply \eqref{claim} to $\gamma_1$ and $\gamma_2$ and conclude that $\gamma_1(t)\in D(\Tilde{\alpha}(t),\varepsilon)$ and $\gamma_2(t+\varepsilon)\in D(\Tilde{\alpha}(t+\varepsilon),\varepsilon)$ for all $t\in[0,L-\varepsilon]$. This immediately implies that $\gamma_1(t) \ll \tilde{\alpha}(t+\tfrac{\eps}{2}) \ll \gamma_2(t+\eps)$ and so $f > 0$. Moreover, as both $\gamma_1(t)$ and $\gamma_2(t+\eps)$ are contained in a common comparison neighborhood, for any fixed $t \in [0,L-\eps]$, the function $f$ is concave.
    
    More explicitly, defining $l_1:=f(L-\varepsilon)$ and $l_2:=f(0)$, we have
    \begin{equation}\label{concavungl}
        f(t)\geq\frac{t}{L-\varepsilon}l_1+\frac{L-\varepsilon-t}{L-\varepsilon}l_2
    \end{equation}
    As geodesics have constant speed, the quantity $l_1$ can also be interpreted as 
    \begin{equation*}
        l_1=\tau(\gamma_1(L-\eps),\gamma_1(L))=L(\gamma_1\big|_{[(i-1)\varepsilon,\,i\varepsilon]})=\frac{L(\gamma_1)}{N} \quad \text{for all } 1\leq i \leq N
    \end{equation*}
    and similarly $l_2=L(\gamma_2\big|_{[(i-1)\varepsilon,\, i\varepsilon]})=\frac{L(\gamma_2)}{N}$. We will now prove that $l_1=l_2$, implying $L(\gamma_1)=L(\gamma_2)$. Suppose to the contrary that $l_2\neq l_1$ and assume w.l.o.g.\ that $l_2 < l_1$. Then, by \eqref{concavungl}, 
    \begin{equation*}
    \tau(\gamma_1(\varepsilon),\gamma_2(2\varepsilon))\geq\frac{\varepsilon}{L-\varepsilon}l_1+\frac{L-2\varepsilon}{L-\varepsilon}l_2>\frac{\varepsilon}{L-\varepsilon}l_2+\frac{L-2\varepsilon}{L-\varepsilon}l_2=l_2.
    \end{equation*}
    In particular, the length of the timelike maximizer $[\gamma_1(\varepsilon),\gamma_2(2\varepsilon)]$ is strictly bigger than $l_2$. Concatenating with $[\gamma_2(0),\gamma_1(\varepsilon)]$ yields a timelike curve connecting $\gamma_2(0)$ to $\gamma_2(2\varepsilon)$ whose length exceeds $2l_2$, contradicting the maximality of $\gamma_2\big|_{[0,2\varepsilon]}$. Hence, indeed, $l_1=l_2$, and we can assume, without loss of generality, that both curves are parametrized by arc length, i.e., $l_1=l_2=\varepsilon$. The concavity relation \eqref{concavungl} then simply reads $f\geq\varepsilon$.

    Following a similar argument, we will now be able to show that $\gamma_1\big|_{[0,\varepsilon]}=\gamma_2\big|_{[0,\varepsilon]}$. For this, fix $t\in[0,\varepsilon]$, and observe that $\gamma_2\big|_{[0,t+\varepsilon]}$ is the \textit{unique} timelike curve from $\gamma_2(0)$ to $\gamma_2(t+\varepsilon)$ of maximal length by \autoref{uniquemaxlocal} and the choice of $\eps$. Now, as $f(t)\geq\varepsilon$, the length of $[\gamma_2(0),\gamma_1(t)]*[\gamma_1(t),\gamma_2(t+\varepsilon)]$ is at least $t+\varepsilon = \tau(\gamma_2(0), \gamma_2(t+\eps))$. Hence, this curve must coincide with $\gamma_2$ on $[0, t + \eps]$, and so, in particular, $\gamma_1(t)=\gamma_2(t)$. Repeating this on the subsequent intervals $[(i-1)\eps, i \eps]$, we get $\gamma_1\big|_{[0,L-\varepsilon]}=\gamma_2\big|_{[0,L-\varepsilon]}$ with $\gamma_1\big|_{[L-\varepsilon,L]}=\gamma_2\big|_{[L-\varepsilon,L]}$ then following immediately from local uniqueness of geodesics. 

    This finishes the proof of \ref{großesThma}. For \ref{großesThmb}, note that, by its very construction, $\gamma$ is always the longest curve from $p$ to $q$ remaining entirely in $U$, and this immediately implies \eqref{Längenungleichung}. Next, we prove \ref{großesThmc}. Given $\delta>0$, we need to find an integer $K\in\mathbb N$ such that for all $k\geq K$ and all $t\in[0,L]$, we have $\gamma_k(t)\in B_\delta(\alpha(t))$, where $\gamma_k$ denotes the unique timelike geodesic from $p_k\in U_0$ to $q_k\in U_L$ and $B_\delta(\dotto)$ denotes a metric ball of radius $\delta >0$ with respect to the background metric $d$. To do so, first recall that, by \autoref{strongcausalityrem}, the collection of timelike diamonds $(D(\tilde\alpha(t),\tilde\varepsilon))_{\tilde\varepsilon>0}$ forms a neighborhood basis of $\alpha(t)$. In particular, it is possible to find $\tilde\varepsilon<\varepsilon$ such that $D(\tilde\alpha(t),\tilde{\varepsilon})\subseteq B_\delta(\alpha(t))$ for all $t\in[0,L]$ and such that $L = \tilde N\tilde \eps$ for some $\tilde N \in \N$. 
    We define $\tilde U:=D(\tilde\alpha,\tilde\varepsilon)$, $\tilde U_0:=D(\tilde\alpha(0),\tilde\varepsilon)$ and $\tilde U_L:=D(\tilde\alpha(L),\tilde\varepsilon)$ then by \ref{großesThma} for any pair $(\tilde p,\tilde q)\in\tilde U_0\times\tilde U_L$ there exists a unique timelike geodesic $\tilde\gamma$ from $\tilde p$ to $\tilde q$ which remains entirely in $\tilde U$.
    As $p_k \rightarrow p$ and $q_k \rightarrow q$, there exists a natural number $K$ so that for any $k\geq K$, we have $p_k\in \tilde U_0$ and $q_k\in \tilde U_L$. Since $\tilde U\subset U$, the unique timelike geodesic from $p_k$ to $q_k$ in $\tilde U$ must coincide with $\gamma_k$ and so by \eqref{claim}, for any $t\in[0,L]$, we have
    \begin{equation*}
        \gamma_k(t)\in D(\tilde{\alpha}(t),\tilde{\varepsilon})\subset B_\delta(\alpha(t)),
    \end{equation*}
    which is what we wanted to show.

    It remains to prove \ref{großesThmd}. First note that in the case $p=\alpha(0)$ and $q=\alpha(L)$, we have $\alpha=\gamma$, so concavity of $\tau(\alpha(t),\gamma(t))$ is trivially satisfied. So assume $p=\alpha(0)$ and $q\gg\alpha(L)$, with all other cases being analogous. We define the function 
    \begin{equation*}
        g\colon [0,L]\times[0,\varepsilon]\to\mathbb{R}, \hspace{1cm} g(t,s)=
        \begin{cases}
            \tau(\alpha(t),\gamma(t+s)), & t\leq L-s\\
            \tau(\alpha(L-s),\gamma(L)), & t>L-s.
        \end{cases}
    \end{equation*}
    Note that $\alpha(t)$ and $\gamma(t+s)$ are both contained in a common comparison neighborhood for any $s\in[0,\varepsilon]$ by \eqref{claim}. This has two implications: First, it yields continuity of $g$ as both curves are continuous and $\tau$ is continuous when restricted to comparison neighborhoods. Second, it means that if for fixed $s$, the function 
    $$g_s\colon (0,L-s]\to\mathbb R,\hspace{1cm} g_s(t):= g(t,s)$$
    is positive then it must be concave on $(0,L-s]$ by local concavity of $\tau$. Let $I\subseteq[0,\varepsilon]$ be the set of values $s$ for which $g_s>0$. If we can show that $0\in I$, we are done. By construction, $\varepsilon\in I$, so $I$ is nonempty. Now, let $s_0\in I\setminus\{0\}$. First note that all $s\in[s_0,\varepsilon]$ must also lie in $I$ since for fixed $t$, the function $t\mapsto g(t,s)$ is monotonically increasing, hence, $I$ is an interval. Next, note that $g_{s_0}(t)$ is bounded from below by $\min\{g_{s_0}(0),g_{s_0}(L-s_0)\}>0$ as a consequence of concavity, so for $s$ close enough to $s_0$, $g_s$ must still be positive and thus, in particular, $I$ is open. Finally, we want to show that $I$ is closed as well. To this end, let $(s_k)_{k\in\mathbb N}$ be a sequence in $I$ such that $s_k\searrow s$ and we argue that $s\in I$. If $t\geq L-s$, then $g_s(t)>0$ is trivially satisfied, so we only need to consider $t< L-s$. Without loss of generality, we assume the values $s_k$ to satisfy $t<L-s_k$ for all $k\in\mathbb N$. Then, by concavity of $g_{s_k}$,
    \begin{equation*}
        g_s(t)=\lim_{k\to\infty}g_{s_k}(t)\geq\lim_{k\to\infty}\frac{g_{s_k}(L-s_k)-g_{s_k}(0)}{L-s_k}\cdot t + g_{s_k}(0)=\frac{g_s(L-s)-g_s(0)}{L-s}\cdot t+g_s(0)>0
    \end{equation*}
    for all $t\in(0,L-s)$. In particular, $s\in I$, so $I$ is also closed and by connectedness coincides with $[0,\varepsilon]$. This implies $g_0(t)>0$ and thus the desired concavity, finishing the proof of \ref{großesThmd}.
\end{proof}

\section{Globalization}\label{sec:Globalization}

As a first application of \autoref{großesThm}, we can prove that if geodesics between any two points in a globally hyperbolic, locally concave regular Lorentzian length space are unique, they must vary continuously. To make this precise, we define the \textit{geodesic map} as in \cite{Patchwork}. In the following denote by $X_\ll:=\{(p,q)\in X\times X\,|\,p\ll q\}$ the set of timelike related points in $X$.

\begin{defin}
    Let $X$ be a regular Lorentzian pre-length space. If each pair of points $(p,q)\in X_\ll$ can be joined by a unique timelike geodesic $\gamma_{pq}$, we say that X is \textit{uniquely geodesic}. In this case, we define the \textit{geodesic map} as
    \begin{equation*}
        G \colon X_\ll\times[0,1]\to X,\hspace{0.5cm}G(x,y,t):=\gamma_{xy}(t)
    \end{equation*}
    and we say that geodesics \textit{vary continuously} in $X$ if $G$ is continuous. 
\end{defin}

\begin{rem}
    Geodesics vary continuously in the above sense if and only if they do so in the sense of uniform convergence w.r.t.\ changes in their endpoints, cf.\ \cite[Prop.\ 4.5]{Patchwork}.
\end{rem}

\begin{cor}\label{contvargeod}
    Let $X$ be a globally hyperbolic, strongly causal, timelike path-connected, locally causally closed, and uniquely geodesic regular Lorentzian pre-length space. If $X$ is locally concave, then geodesics vary continuously in $X$. 
\end{cor}

\begin{proof}
Let $p\ll q$ be in $X$ and let $(p_k)_{k\in\mathbb N}$, $(q_k)_{k\in\mathbb N}$ be sequences of points satisfying $p_k\ll q_k$ for all $k\in\mathbb N$ and $p_k\to p$ and $q_k\to q$. Choose neighborhoods $U_0$ of $p$ and $U_L$ of $q$ as in \autoref{großesThm}, then for $k$ larger than some $K\in\mathbb N$ we have $p_k\in U_0$ and $q_k\in U_L$ and by \autoref{großesThm}\ref{großesThmc} the curves $\gamma_{p_kq_k}$ converge to $\gamma_{pq}$ uniformly. By the above remark, this is enough to prove the corollary.
\end{proof}

While this result serves as a first step, we do not want to assume our space to be uniquely geodesic but would rather have this be a consequence of concavity. The difficulty of concluding this from \autoref{großesThm} is that the uniqueness statement there only applies to curves remaining in one specific neighborhood. The critical step to go from this to a global statement is to get rid of the neighborhood of $\alpha$ while still recovering the essence of \autoref{großesThm}. The following statement and proof are roughly based on the ideas and constructions in \cite[4.2 Lem.]{BallmannLectures}, where a similar result is proven for metric spaces. 

\begin{lem}\label{globlem}
    Let $X$ be a globally hyperbolic, strongly causal, timelike path-connected and locally causally closed regular Lorentzian pre-length space and let $\gamma\colon [0,1]\to X$ be a future-directed timelike geodesic. Let ${D}_0$ and ${D}_1$ be timelike comparison diamonds around $\gamma(0)$ and $\gamma(1)$, respectively.
    If $X$ is locally concave, then for $\tilde{D}_0:=D_0\cap I^-(\gamma(0))$ and $\tilde{D}_1:=D_1\cap I^+(\gamma(1))$ the following holds:
    \begin{enumerate}[label=(\roman*)]
    \item There exists a continuous map 
    $\Gamma\colon \tilde{D}_0\times \tilde{D}_1\times[0,1]\to X$
    such that $\Gamma(p,q,\dotto)$ is a future-directed timelike geodesic from $p$ to $q$ for any pair $(p,q) \in \tilde D_0 \times \tilde D_1$. \label{globlem1}
    \item If $H\colon [0,1]\times[0,1]\to X$ is a homotopy with $H(0,\dotto)=\gamma$, $H(s,0)\in \tilde{D}_0$, $H(s,1)\in \tilde{D}_1$ and $H(s,\dotto)$ is a timelike geodesic for all $s\in[0,1]$, then $H(s,\dotto)=\Gamma(H(s,0),H(s,1),\dotto)$.\label{globlem2}
    \end{enumerate}
\end{lem}

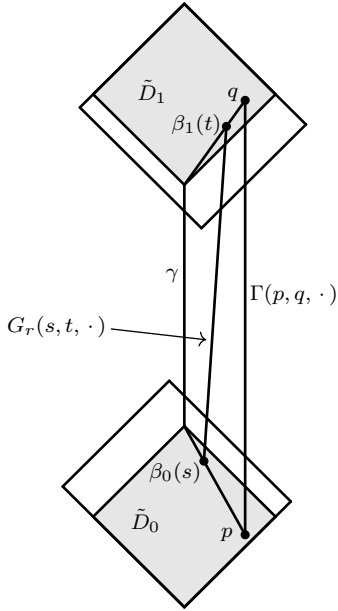
\begin{figure}
\centering
\begin{tikzpicture}[line cap=round,line join=round,x=0.8cm,y=0.8cm]
\clip(-3.5,-3.5) rectangle (3.5,7.5);
\fill[line width=1.pt,fill=black,fill opacity=0.10000000149011612] (0.,4.) -- (1.5,5.5) -- (0.,7.) -- (-1.5,5.5) -- cycle;
\fill[line width=1.pt,fill=black,fill opacity=0.10000000149011612] (0.,0.) -- (-1.5,-1.5) -- (0.,-3.) -- (1.5,-1.5) -- cycle;
\draw [line width=1.pt] (0.,0.)-- (0.,4.);
\draw [line width=1.pt] (2.,5.)-- (0.,7.);
\draw [line width=1.pt] (0.,-3.)-- (-2.,-1.);
\draw [line width=1.pt] (0.,4.)-- (1.5,5.5);
\draw [line width=1.pt] (1.5,5.5)-- (0.,7.);
\draw [line width=1.pt] (0.,7.)-- (-1.5,5.5);
\draw [line width=1.pt] (-1.5,5.5)-- (0.,4.);
\draw [line width=1.pt] (0.,0.)-- (-1.5,-1.5);
\draw [line width=1.pt] (-1.5,-1.5)-- (0.,-3.);
\draw [line width=1.pt] (0.,-3.)-- (1.5,-1.5);
\draw [line width=1.pt] (1.5,-1.5)-- (0.,0.);
\draw [line width=1.pt] (0.,0.)-- (1.,-1.8);
\draw [line width=1.pt] (0.,4.)-- (1.,5.4);
\draw [line width=1.pt] (1.,-1.8)-- (1.,5.4);
\draw [line width=1.pt] (0.32178169325697603,-0.5792070478625568)-- (0.6916095761375537,4.968253406592575);
\draw [line width=1.pt] (1.7512949797671777,-1.2487050202328223)-- (-0.24870502023282226,0.7512949797671777);
\draw [line width=1.pt] (-0.24870502023282226,0.7512949797671777)-- (-2.,-1.);
\draw [line width=1.pt] (1.7512949797671777,-1.2487050202328223)-- (0.,-3.);
\draw [line width=1.pt] (-1.7184005936154203,5.281599406384579)-- (0.,7.);
\draw [line width=1.pt] (-1.7184005936154203,5.281599406384579)-- (0.2815994063845796,3.2815994063845797);
\draw [line width=1.pt] (0.2815994063845796,3.2815994063845797)-- (2.,5.);
\draw [->,line width=0.5pt] (-1.2,1.6107940113352153) -- (0.3612952690805856,1.4315144091321663);
\begin{scriptsize}
\draw[color=black] (-0.2,2.5) node {$\gamma$};
\draw[color=black] (-0.5282484855600129,5.5503330916580556) node {$\tilde D_1$};
\draw[color=black] (-0.6376870376821366,-1.5631727962800157) node {$\tilde D_0$};
\draw [fill=black] (1.,-1.8) circle (1.5pt);
\draw[color=black] (0.7,-1.8) node {$p$};
\draw [fill=black] (1.,5.4) circle (1.5pt);
\draw[color=black] (0.8,5.5) node {$q$};
\draw[color=black] (1.8,2.1796256862966312) node {$\Gamma(p,q,\dotto)$};
\draw [fill=black] (0.32178169325697603,-0.5792070478625568) circle (1.5pt);
\draw[color=black] (-0.15,-0.77) node {$\beta_0(s)$};
\draw [fill=black] (0.6916095761375537,4.968253406592575) circle (1.5pt);
\draw[color=black] (0.2,5) node {$\beta_1(t)$};
\draw[color=black] (-2.1,1.65) node {$G_r(s,t,\dotto)$};
\end{scriptsize}
\end{tikzpicture}
\caption{Setup for the proof of \autoref{globlem}}
\end{figure}\label{figure:globlem}

\begin{proof}
    \ref{globlem1} Let $p\in \tilde{D}_0$ and $q\in \tilde{D}_1$ and let $\beta_0,\beta_1\colon [0,1]\to X$ be the unique timelike geodesics from $p$ to $\gamma(0)$ and from $\gamma(1)$ to $q$, respectively. We introduce a set $A \subseteq [0,1]$ consisting of all $r\in[0,1]$ for which there exists a continuous map $G_r\colon [0,r]\times[0,r]\times [0,1]\to X$ with $G_r(0,0,\dotto)=\gamma$ and such that $G_r(s,t,\dotto)$ is a geodesic from $\beta_0(s)$ to $\beta_1(t)$ for all $s,t \leq r$, see. \autoref{figure:globlem}.
    
    Our aim is to show that $A=[0,1]$. Clearly, $0\in A$. Now note that if $r\in A$, then also $[0,r)\subseteq A$. Consequently, $A$ is an interval. 
    Moreover, for any $r\in A\setminus \{1\}$, we can invoke \autoref{großesThm} to uniquely and continuously extend the function $G_r$ to the domain $[0,r+\varepsilon]\times[0,r+\varepsilon]\times[0,1]$ for some $\varepsilon>0$ small enough. This implies that for $r_1<r_2\in A$, the functions $G_{r_1}$ and $G_{r_2}$ must coincide on their common domain, and it shows that the set $A$ is open. We now argue that it is also closed.
    
    To this end, we set $r_0:=\sup A$ and show $r_0\in A$. We start by defining a map $g_{r_0}\colon[0,r_0)\times[0,r_0)\times[0,1]\to X$ which agrees with $G_r$ for all $r<r_0$. Our aim is to find a continuous extension of this map to $[0,r_0]\times[0,r_0]\times[0,1]$ such that the resulting new curves are again geodesics. If we fix $s\in[0,r_0)$ and choose a sequence $t_k\nearrow r_0$, we know that the timelike geodesics $g_{r_0}(s,t_k,\dotto)$ converge uniformly to a timelike curve $\alpha$ along some subsequence of $(t_k)_{k \in \N}$ by the limit curve theorem \cite[Thm.\ 3.7]{LLS}. In fact, \autoref{geodesiclimit} ensures that $\alpha$ is even a geodesic, but it is not yet clear whether the convergence holds along the entire sequence. To see that it does, choose a neighborhood $U$ of $\alpha$ and a neighborhood $U_1$ of $\alpha(1)$ as in \autoref{großesThm} and pick $K\in\mathbb N$ large enough so that $g_{r_0}(s,t_K,\dotto)\subseteq U$ and $g_{r_0}(s,t_K,1)=\beta_1(t_K)\in U_1$. Since $U_1$ is causally convex, we know that $\beta_1(t)\in U_1$ for all $t\in[t_K,r_0]$ and hence, for all such $t$, there exists a unique timelike geodesic from $\beta_0(s)$ to $\beta_1(t)$ which remains in $U$. We have already seen that these geodesics vary continuously and that they converge uniformly to $\alpha$ as $t\nearrow r_0$. It remains to show that they coincide with $g_{r_0}(s,t,\dotto)$. Suppose the contrary, and let $\overline{t}$ be the infimum of the set of all $t$ on which they disagree. Applying \autoref{großesThm} to $g_{r_0}(s,\overline{t},\dotto)$ we find a neighborhood of this curve in which geodesics between $\beta_0(s)$ and $\beta_1(t)$ are unique, which is a contradiction. 

    Repeating this for all $s\in[0,r_0)$, we can extend $g_{r_0}$ to $[0,r_0]\times[0,r_0)\times[0,1]$ and it is a simple application of \autoref{großesThm} that this extension is continuous. Now fix $t\in[0,r_0)$ and choose a sequence $s_k\nearrow r_0$ to define a geodesic from $\beta_0(r_0)$ to $\beta_1(t)$ as the limit of $g_{r_0}(s_k,t,\dotto)$ as above. Taken together, this allows us to continuously extend $g_{r_0}$ to $[0,r_0]\times[0,r_0]\times[0,1]$ implying that $r_0\in A$ and hence $A$ is closed.

    This proves that the map $G_1$ is defined and we set $\Gamma(p,q,\dotto):=G_1(1,1,\dotto)$. To finish the proof of \ref{globlem1}, we have to show that this assignment is continuous, i.e. for any $(p,q)\in\tilde D_0\times\tilde D_1$ and any sequence $(p_k,q_k)\to (p,q)$ we have that $\Gamma(p_k,q_k,\dotto)\to\Gamma(p,q,\dotto)$ uniformly. For this, we again consider the neighborhoods $U,U_0,U_1$ of the curve $\Gamma(p,q,\dotto)$ as well as its endpoints $p$ and $q$ introduced in \autoref{großesThm}. By \autoref{großesThmc} of the same theorem, we know that the unique geodesics $\gamma_{p_kq_k}$ joining $p_k$ to $q_k$ within $U$ converge uniformly to $\Gamma(p,q,\dotto)$, so if we can show that $\gamma_{p_kq_k}=\Gamma(p_k,q_k,\dotto)$ for $k$ large enough, we are done. To this end, we define neighborhoods $U^t,U_0^t,U_1^t$ as in \autoref{großesThm} for all pairs of points $(\beta_0(t),\beta_1(t))$ and the respective geodesic $\Gamma(\beta_0(t),\beta_1(t),\dotto)$ between them.
    We set $V_0:=\bigcup_{t} U_0^t, \,V_1:=\bigcup_{t} U_1^t$ and note that they are open neighborhoods of $\beta_0$ and $\beta_1$. Recalling that geodesics vary continuously in $D_0$ and $D_1$ by \autoref{contvargeod}, we conclude that, for any fixed $k$ large enough, the geodesics $\tilde{\beta}_0$ and $\tilde{\beta}_1$ connecting ${p}_k$ and $q_k$ to the endpoints of $\gamma$ do not leave $V_0$ and $V_1$, respectively. Hence, for each ${t}\in[0,1]$ there exists an $s\in[0,1]$ such that $\tilde{\beta}_0({t})\in U^s_0$ and $\tilde{\beta}_1({t})\in U_1^s$. We set 
    \begin{equation*}
        J^s:=\{t \in [0,1] \, \vert \, \tilde{\beta}_0({t})\in U^s_0, \,\tilde{\beta}_1({t})\in U_1^s \}.
    \end{equation*}
    Now, by \autoref{großesThm}, for each $t \in J^s$, there exists a unique geodesic $\tilde{\gamma}_t^s$ connecting $\tilde{\beta}_0(t)$ to $\tilde{\beta}_1(t)$ which is entirely contained in $U^s$. A priori, $\tilde{\gamma}_t^s$ need not coincide with $\Gamma(\tilde{\beta}_0(t),\tilde{\beta}_1(t),\dotto)$. However, similarly to what we have done before, one can show that the set $I^s$ of parameter values $t \in [0,1]$ on which they do coincide is connected (by causal convexity of $U_0^s$ and $U_1^s$), open and closed as a subset of $J^s$ and must therefore either be empty or agree with $J^s$. Now, as for $t=s=0$, we have $\tilde{\gamma}^0_0=\gamma=\Gamma(\tilde{\beta}_0(0),\tilde{\beta}_1(0),\dotto)$, we see that $I^0 \neq \varnothing$. Hence, $I^s=J^s$ for all $s$ such that $J^s\cap J^0 \neq \varnothing$ and, by iteration, $I^s=J^s$ for all $s$ for which $\tilde{\beta}_0$ and $\tilde{\beta}_1$ pass through $U^s_0$ and $U^s_1$. In particular, the curves coincide in $U$, the neighborhood of $\Gamma(p,q,\dotto)$. Consequently, we have $\gamma_{p_kq_k}=\Gamma(p_k,q_k,\dotto)$ which is what we needed to prove continuity of $\Gamma$.

    Finally, item \ref{globlem2} follows analogously by showing that the set of all $s\in[0,1]$ for which $H(s,\dotto)$ and $\Gamma(H(s,0),H(s,1),\dotto)$ coincide is nonempty, open and closed, again invoking \autoref{großesThm}. 
\end{proof}

Without further constraints on our spaces, there is no hope of proving uniqueness of geodesics between arbitrary timelike related points. Indeed, the simple example of a cylinder, whose curvature is clearly bounded from above by zero, but on which we can easily find timelike related points connected by two 
geodesics, shows that an additional topological assumption is necessary.

\begin{defin}
    Let $X$ be a Lorentzian pre-length space and let $\gamma_0,\gamma_1\colon [0,1]\to X$ be two future-directed timelike curves. A \textit{timelike homotopy} between $\gamma_0$ and $\gamma_1$ is a continuous map
    \begin{equation*}
        H\colon[0,1]\times[0,1]\to X
    \end{equation*}
    with $H(0,\dotto)=\gamma_0$, $H(1,\dotto)=\gamma_1$ and such that $H(s,\dotto)$ is future-directed timelike for any $s\in[0,1]$. If $X$ is timelike path-connected and any pair of future-directed timelike curves with common endpoints are timelike path-homotopic, then $X$ is called \emph{future one-connected}.
\end{defin}

The notion of future one-connectedness is logically independent of simple connectedness, see the discussion in \cite[Section 10.1]{BeemEhrlich}. 

Now, if two timelike related points in a future one-connected space are joined by two distinct geodesics, then, by definition, we can find a homotopy between those geodesics consisting of timelike curves. Using \autoref{globlem}, we will now show that this homotopy can be assumed to consist only of \emph{geodesics}. In fact, we have the following even more general statement:

\begin{prop}\label{geodhomotop}
    Let $X$ be a globally hyperbolic, strongly causal, timelike path-connected and locally causally closed regular Lorentzian pre-length space, and let $H_0\colon[0,1]\times[0,1]\to X$ be a timelike homotopy between two future-directed timelike curves $\gamma_0$ and $\gamma_1$. If $X$ is locally concave, then there exists a timelike homotopy $H_1$ consisting of future-directed timelike geodesics $H_1(s,\dotto)$ between $H_0(s,0)$ and $H_0(s,1)$ such that $H_0$ is homotopic to $H_1$ via a timelike homotopy $G$. Moreover, if $\gamma_0$ and $\gamma_1$ are geodesics, then $\gamma_0=H_0(0,\dotto)=H_1(0,\dotto)$ and $\gamma_1=H_0(1,\dotto)=H_1(1,\dotto)$.
\end{prop}

\begin{figure}
\begin{tikzpicture}[line cap=round,line join=round,x=0.35cm,y=0.35cm]
\clip(-4.,-11.) rectangle (17.,11.);
\fill[line width=1.pt,fill=black,fill opacity=0.10000000149011612] (8.,-10.) -- (12.,-6.) -- (6.,0.) -- (2.,-4.) -- cycle;
\fill[line width=1.pt,fill=black,fill opacity=0.10000000149011612] (4.,-4.) -- (9.,1.) -- (5.,5.) -- (0.,0.) -- cycle;
\fill[line width=1.pt,fill=black,fill opacity=0.10000000149011612] (6.,2.) -- (3.,5.) -- (8.,10.) -- (11.,7.) -- cycle;
\draw [shift={(8.,0.)},line width=1.pt]  plot[domain=2.356194490192345:3.9269908169872414,variable=\t]({1.*11.313708498984761*cos(\t r)+0.*11.313708498984761*sin(\t r)},{0.*11.313708498984761*cos(\t r)+1.*11.313708498984761*sin(\t r)});
\draw [line width=1.pt] (0.,8.)-- (16.,8.);
\draw [line width=1.pt] (0.,-8.)-- (16.,-8.);
\draw [shift={(24.,0.)},line width=1.pt]  plot[domain=2.356194490192345:3.9269908169872414,variable=\t]({1.*11.313708498984761*cos(\t r)+0.*11.313708498984761*sin(\t r)},{0.*11.313708498984761*cos(\t r)+1.*11.313708498984761*sin(\t r)});
\draw [shift={(16.,0.)},line width=1.pt]  plot[domain=2.356194490192345:3.9269908169872414,variable=\t]({1.*11.313708498984761*cos(\t r)+0.*11.313708498984761*sin(\t r)},{0.*11.313708498984761*cos(\t r)+1.*11.313708498984761*sin(\t r)});
\draw [line width=1.pt] (8.,-10.)-- (12.,-6.);
\draw [line width=1.pt] (12.,-6.)-- (6.,0.);
\draw [line width=1.pt] (6.,0.)-- (2.,-4.);
\draw [line width=1.pt] (2.,-4.)-- (8.,-10.);
\draw [line width=1.pt] (4.,-4.)-- (9.,1.);
\draw [line width=1.pt] (9.,1.)-- (5.,5.);
\draw [line width=1.pt] (5.,5.)- - (0.,0.);
\draw [line width=1.pt] (0.,0.)-- (4.,-4.);
\draw [line width=1.pt] (6.,2.)-- (3.,5.);
\draw [line width=1.pt] (3.,5.)-- (8.,10.);
\draw [line width=1.pt] (8.,10.)-- (11.,7.);
\draw [line width=1.pt] (11.,7.)-- (6.,2.);
\begin{scriptsize}
\draw[color=black] (-2.6369389155826764,0.6897111289998015) node {$\gamma_0$};
\draw[color=black] (13.265871261256824,0.6897111289998015) node {$\gamma_1$};
\draw[color=black] (5.28713304277594,0.6897111289998015) node {$\gamma_{\bar s}$};
\draw[color=black] (7.46315255690709,-4.814336436694856) node {$D_0$};
\draw[color=black] (2.865779788767057,0.16304242961114095) node {$D_1$};
\draw[color=black] (8.5,6.501762353594048) node {$D_2$};
\end{scriptsize}
\end{tikzpicture}
\begin{tikzpicture}[line cap=round,line join=round,x=0.35cm,y=0.35cm]
\clip(-4.,-11.) rectangle (17.,11.);
\fill[line width=1.pt,fill=black,fill opacity=0.10000000149011612] (4.453777838068226,-6.453777838068226) -- (5.869797834001266,-5.037757842135186) -- (9.416019995933041,-8.583980004066959) -- (8.,-10.) -- cycle;
\fill[line width=1.pt,fill=black,fill opacity=0.10000000149011612] (8.,-8.) -- (11.,-5.) -- (6.,0.) -- (3.,-3.) -- cycle;
\fill[line width=1.pt,fill=black,fill opacity=0.10000000149011612] (4.913835442747027,-2.257643773828816) -- (1.3280958344591058,1.3280958344591058) -- (5.,5.) -- (8.585739608287922,1.4142603917120784) -- cycle;
\fill[line width=1.pt,fill=black,fill opacity=0.10000000149011612] (5.298482978732162,3.671176030036526) -- (9.813653474347818,8.186346525652182) -- (8.,10.) -- (3.484829504384344,5.484829504384344) -- cycle;
\draw [shift={(8.,0.)},line width=1.pt]  plot[domain=2.356194490192345:3.9269908169872414,variable=\t]({1.*11.313708498984761*cos(\t r)+0.*11.313708498984761*sin(\t r)},{0.*11.313708498984761*cos(\t r)+1.*11.313708498984761*sin(\t r)});
\draw [line width=1.pt] (0.,8.)-- (16.,8.);
\draw [line width=1.pt] (0.,-8.)-- (16.,-8.);
\draw [shift={(24.,0.)},line width=1.pt]  plot[domain=2.356194490192345:3.9269908169872414,variable=\t]({1.*11.313708498984761*cos(\t r)+0.*11.313708498984761*sin(\t r)},{0.*11.313708498984761*cos(\t r)+1.*11.313708498984761*sin(\t r)});
\draw [shift={(16.,0.)},line width=1.pt]  plot[domain=2.356194490192345:3.9269908169872414,variable=\t]({1.*11.313708498984761*cos(\t r)+0.*11.313708498984761*sin(\t r)},{0.*11.313708498984761*cos(\t r)+1.*11.313708498984761*sin(\t r)});
\draw [line width=1.pt] (4.453777838068226,-6.453777838068226)-- (5.869797834001266,-5.037757842135186);
\draw [line width=1.pt] (5.869797834001266,-5.037757842135186)-- (9.416019995933041,-8.583980004066959);
\draw [line width=1.pt] (9.416019995933041,-8.583980004066959)-- (8.,-10.);
\draw [line width=1.pt] (8.,-10.)-- (4.453777838068226,-6.453777838068226);
\draw [line width=1.pt] (8.,-8.)-- (11.,-5.);
\draw [line width=1.pt] (11.,-5.)-- (6.,0.);
\draw [line width=1.pt] (6.,0.)-- (3.,-3.);
\draw [line width=1.pt] (3.,-3.)-- (8.,-8.);
\draw [line width=1.pt] (4.913835442747027,-2.257643773828816)-- (1.3280958344591058,1.3280958344591058);
\draw [line width=1.pt] (1.3280958344591058,1.3280958344591058)-- (5.,5.);
\draw [line width=1.pt] (5.,5.)-- (8.585739608287922,1.4142603917120784);
\draw [line width=1.pt] (8.585739608287922,1.4142603917120784)-- (4.913835442747027,-2.257643773828816);
\draw [line width=1.pt] (5.298482978732162,3.671176030036526)-- (9.813653474347818,8.186346525652182);
\draw [line width=1.pt] (9.813653474347818,8.186346525652182)-- (8.,10.);
\draw [line width=1.pt] (8.,10.)-- (3.484829504384344,5.484829504384344);
\draw [line width=1.pt] (3.484829504384344,5.484829504384344)-- (5.298482978732162,3.671176030036526);
\draw [line width=1.pt] (8.,-8.)-- (4.913835442747027,-2.257643773828816);
\draw [line width=1.pt] (8.,-8.)-- (5.298482978732162,3.671176030036526);
\draw [line width=1.pt] (8.,-8.)-- (8.,8.);
\draw [->,line width=0.5pt] (10,-2.5)-- (6,-4);
\draw [->,line width=0.5pt] (10,0)-- (6,1.5);
\draw [->,line width=0.5pt] (9.7,4.1)-- (8.1,4.1);
\begin{scriptsize}
\draw[color=black] (-2.6369389155826764,0.6897111289998015) node {$\gamma_0$};
\draw[color=black] (13.265871261256824,0.6897111289998015) node {$\gamma_1$};
\draw[color=black] (3.7,-6.5) node {$\tilde D_0^-$};
\draw[color=black] (2.3,-3) node {$\tilde D_0^+$};
\draw[color=black] (0.6,1.3) node {$\tilde D_1$};
\draw[color=black] (2.794406243882176,5.5) node {$\tilde D_2$};
\draw[color=black] (10.5,-2.3) node {$\alpha_{r_0\bar s}$};
\draw[color=black] (10.5,-0.4) node {$\alpha_{r_1\bar s}$};
\draw [fill=black] (8.,8.) circle (1.5pt);
\draw[color=black] (10.5,4) node {$\alpha_{1\bar s}$};
\draw [fill=black] (4.913835442747027,-2.257643773828816) circle (1.5pt);
\draw [fill=black] (5.298482978732162,3.671176030036526) circle (1.5pt);
\draw [fill=black] (8,-8) circle (1.5pt);
\end{scriptsize}
\end{tikzpicture}
\caption{Setup for the proof of \autoref{geodhomotop}}
\end{figure}
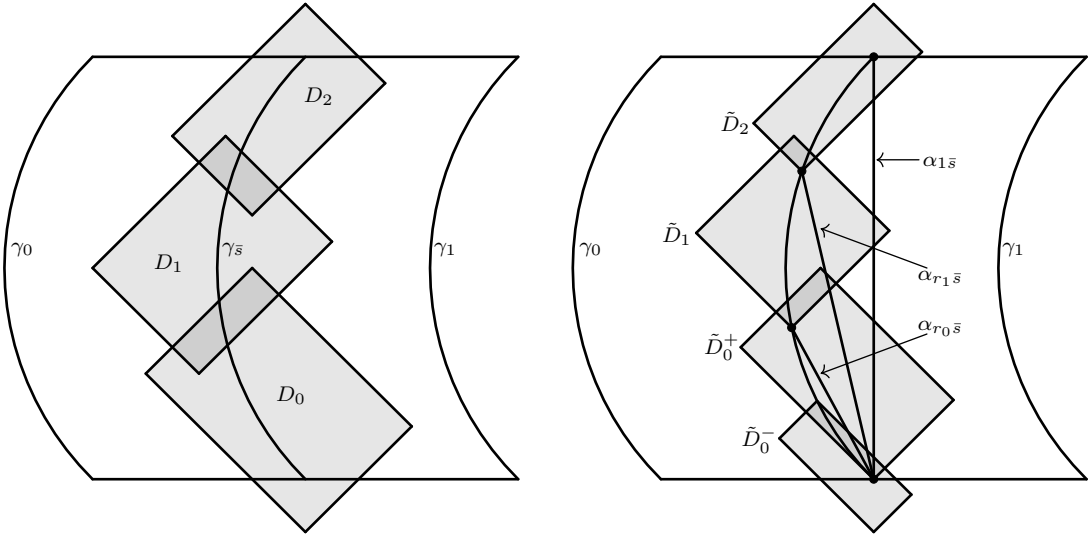\label{figure:geodhomotop}

\begin{proof}
    Cover the compact image of the homotopy $H_0$ with finitely many timelike comparison diamonds, recalling that these form a basis of the topology of $X$ by \autoref{strongcausalityrem}, and denote the collection of these sets by $\mathcal{D}:=\{D_0, \dots, D_n\}$. Now, fix $\bar s\in[0,1]$ and assume, without loss of generality, that $\mathcal{D}_{\bar s}=\{D_0,\ldots,D_k\}$ with $k\leq n$ is a minimal subcollection of $\mathcal{D}$ which still covers the image of $\gamma_{\bar s}:=H_0(\bar s,\dotto)$. We will construct the homotopy $H_1$ using \autoref{globlem}, but first we need a finer cover of $\gamma_{\bar s}$ to do so. Without loss of generality, let $\gamma_{\bar s}(0)\in D_0$ and pick $r_0\in(0,1]$ such that $\gamma_{\bar s}(r_0)\in D_0$. Now, define
    \begin{equation*}
        \widetilde{D}_0^-:=I^-(\gamma_{\bar s}(r_0))\cap D_0\hspace{1cm}\text{and}\hspace{1cm}\widetilde{D}_0^+:=I^+(\gamma_{\bar s}(0))\cap D_0.
    \end{equation*}
    Next, we pick $r_1>r_0$ such that $\gamma_{\bar s}(r_1)\in D_0\cap D_1$ and define
    \begin{equation*}
        \widetilde{D}_1:= I^+(\gamma_{\bar s}(r_1))\cap D_1.
    \end{equation*}
    We continue this procedure of picking $r_i>r_{i-1}$ such that $\gamma_{\bar s}(r_i)\in D_{i-1}\cap D_i$ and successively define $\widetilde{D}_i:= I^+(\gamma_{\bar s}(r_i))\cap D_i$. Thereby, we obtain a new open cover $\widetilde{\mathcal{D}}_{\bar s}=\{\widetilde{D}_0^-,\widetilde{D}_0^+,\widetilde{D}_1,\ldots,\widetilde{D}_k\}$ of $\gamma_{\bar s}$ and, moreover, of any curve $\gamma_s$ with $s$ close enough to $\bar s$.   
    Now, whenever the parameters $s$ and $r$ are such that $\gamma_s(0)\ll\gamma_s(r)$ and both points contained in $D_0$, there exists a unique timelike geodesic $\alpha_{rs}$ connecting the two points and these geodesics vary continuously in $s$ and $r$ by \autoref{uniquemaxlocal} and \autoref{contvargeod}. We want to continue this construction beyond $D_0$. To this end, denote by $\eta_1$ the unique geodesic joining $\gamma_{\bar s}(r_0)$ to $\gamma_{\bar s}(r_1)$ (both of which lie in a common comparison neighborhood $D_0$) and apply \autoref{globlem} to obtain a family of continuously varying geodesics between any pair of points in $\widetilde{D}_0^-\times \widetilde{D}_1$. In particular, we can define the geodesic $\alpha_{rs}$ connecting $\gamma_s(0)$ to $\gamma_s(r)$ whenever these points lie in the respective neighborhoods. We now continue this procedure as follows: If we have already shown that geodesics between points in $\widetilde{D}_0^-$ and $\widetilde{D}_{i-1}$ exist and vary continuously with their endpoints, we fix a reference geodesic $\eta_i$ from $\gamma_{\bar s}(0)$ to $\gamma_{\bar s}(r_i)$ and apply \autoref{globlem} to obtain the same result for pairs of points in $\widetilde{D}_0^- \times \widetilde{D}_i$. Note that if the point $\gamma_s(r)$ lies in the intersection $\widetilde{D}_{i-1}\cap\widetilde{D}_i$ then the choice of a geodesic to this point is independent of whether we regard it as an element of $\widetilde{D}_{i-1}$ or $\widetilde{D}_i$ by \autoref{globlem}\ref{globlem2}. Hence, for all $\gamma_s$ covered by $\widetilde{\mathcal{D}}_{\bar s}$ we can define continuously varying geodesic $\alpha_{rs}$ from $\gamma_s(0)$ to $\gamma_s(r)$, see \autoref{figure:geodhomotop}.

    Observe that this construction is independent of the choice of parameters $r_i$ and of the collection of sets $\mathcal{D}_{\bar s}$, as can easily be proven using \autoref{globlem}\ref{globlem2}. Moreover, for $s\neq\bar s$, the definition of the curves $\alpha_{rs}$ does not depend on the fact that the sets $\widetilde{D}_i=I^+(\gamma_{\bar s}(r_i))\cap D_i$ were always defined using the same fixed $\bar s$ either. Indeed, one could instead pick parameters $r_i'$ such that $\gamma_s(r_i')\in \widetilde{D}_{i-1}\cap\widetilde{D}_i$ and replace $\widetilde{D}_i$ with $\widetilde{D}'_i:=I^+(\gamma_s(r_i'))\cap \widetilde D_i$. Then on $\widetilde{D}'_i$ the two a priori independent constructions of $\alpha_{rs}$ must coincide by \autoref{globlem}\ref{globlem2}.

    Finally, we set $G(r,s,\dotto):=\alpha_{rs}*\gamma_s\big|_{[r,1]}$. This is continuous by construction, and $H_1(s,t):=G(1,s,t)$ is the timelike geodesic homotopy between $\gamma_0$ and $\gamma_1$ we were after. All other claims immediately follow from the construction of $G$. \qedhere
    
\end{proof}

\begin{prop}\label{globalconcav}
    Let $X$ be a globally hyperbolic regular Lorentzian length space. If $X$ is locally concave and every pair of timelike related points $x,y \in X$ can be joined by a unique timelike geodesic $\gamma_{x,y}$, then:
    \begin{enumerate}[label=(\roman*)]
        \item Each timelike geodesic $\gamma_{x,y}$ is length-maximizing, and \label{globalconcav1}
        \item\label{globalconcav2} $X$ is globally concave, i.e., it serves as a comparison neighborhood in the sense of \autoref{def:Concavity}
    \end{enumerate}
\end{prop}

\begin{proof}
    \ref{globalconcav1} By the generalized Avez-Seifert Theorem \cite[Thm.\ 3.30]{LLS}, there exists a maximizing curve between every pair of timelike related points. Since, by assumption, timelike geodesics are unique, any such geodesic must be maximizing. 
    
    \ref{globalconcav2} We need to check items \ref{concav1}-\ref{concav3} in \autoref{def:Concavity}. The first item, i.e., continuity of $\tau$, holds by global hyperbolicity (see \cite[Thm.\ 3.28]{LLS}), while the second one follows from the above. It remains to prove the concavity of the time separation function. For any pair of timelike related points $x,y\in X$, denote by $\gamma_{x,y}\colon [0,1]\to X$ the unique timelike maximizer from $x$ to $y$. We start off by proving 
    \begin{equation}\label{globalconcaveq}
        \tau(\gamma_{x,y_0}(t), \gamma_{x,y_1}(t)) \geq t \tau(y_0, y_1), \hspace{1cm}t\in[0,1],
    \end{equation}
    for all $x,y_0,y_1\in X$ with $y_0\ll y_1$.

    To do so, we connect $y_0$ and $y_1$ by a timelike maximizer $\beta\colon [0,1]\to X$ 
    and consider the homotopy 
    \begin{equation*}
    H\colon [0,1]\times[0,1]\to X\hspace{1cm}(s,t)\mapsto\gamma_{x,\beta(s)}(t).
    \end{equation*}
    For any fixed $s\in[0,1]$, \autoref{großesThm} yields neighborhoods $U^s$ of $\gamma_{x,\beta(s)}$ and $U^s_1$ of $\beta(s)$. Clearly, $(U_1^{s})_{s\in[0,1]}$ covers $\beta([0,1])$ and since this set is compact, there are finitely many parameters $0=s_0<s_1<\cdots<s_n=1$ such that $(U_1^{s_i})_{0\leq i\leq n}$ is still a cover. Note that $H([0,1],[0,1])$ is thereby covered by $(U^{s_i})_{0\leq i\leq n}$ since $\beta(s)\in U^{s_i}_1$ implies $\gamma_{x,\beta(s)}(t)\in U^{s_i}$. For $1\leq i\leq n$ pick parameters $\lambda_i\in[0,1]$ such that $\beta(\lambda_i)\in U_1^{s_{i-1}}\cap U_1^{s_i}$ and set $\lambda_0=0$, $\lambda_{n+1}=1$. Then for $1\leq i\leq n-1$, both the curve $\gamma_{x,\beta(\lambda_i)}$ and the curve $\gamma_{x,\beta(\lambda_{i+1})}$ are contained in $U^{s_i}$ so we can apply \autoref{großesThm}\ref{großesThmd} to infer
    \begin{equation*}
        \tau(\gamma_{x,\beta(\lambda_i)}(t),\gamma_{x,\beta(\lambda_{i+1})}(t))\geq t\tau(\beta(\lambda_i),\beta(\lambda_{i+1})).
    \end{equation*}
    Summing up all these inequalities, we get 
    \begin{equation*}
        \tau(\gamma_{x,y_0}(t), \gamma_{x,y_1}(t)) \geq \sum_{i=0}^n \tau(\gamma_{x,\beta(\lambda_i)}(t), \gamma_{x,\beta(\lambda_{i+1})}(t)) \geq t  \sum_{i=0}^n \tau(\beta(\lambda_i), \beta(\lambda_{i+1})) = t \tau(y_0, y_1),
    \end{equation*}
    which gives \eqref{globalconcaveq}. Analogously, we can prove $\tau(\gamma_{x_0,y}(t),\gamma_{x_1,y}(t))\geq t\tau(x_0,x_1)$ for $x_0\ll x_1$. For the general case with $x_0\ll x_1$ and $y_0\ll y_1$, the previous cases yield
    \begin{align*}
        \tau(\gamma_{x_0,y_0}(t),\gamma_{x_1,y_1}(t))&\geq\tau(\gamma_{x_0,y_0}(t),\gamma_{x_0,y_1}(t))+\tau(\gamma_{x_0,y_1}(t),\gamma_{x_1,y_1}(t))\\
        &\geq t\tau(y_0,y_1)+(1-t)\tau(x_0,x_1),
    \end{align*}
    where in the second summand, concavity was applied to the time reversals of the curves $\gamma_{x_0,y_1}$ and $\gamma_{x_1,y_1}$.
\end{proof}

The results of this section can now be summarized by the following theorem:

\begin{thm}[Synthetic Lorentzian Cartan-Hadamard Theorem]\label{C-H}
    Let $X$ be a globally hyperbolic, strongly causal, and locally causally closed regular Lorentzian pre-length space. If $X$ is locally concave and future one-connected, then the following holds:
    \begin{enumerate}[label=(\roman*)]
        \item Every pair of timelike related points $x\ll y$ is connected by a unique geodesic.\label{C-H1}
        \item Geodesics vary continuously with their endpoints.\label{C-H2}
    \end{enumerate}
    Moreover, if $X$ is a Lorentzian length space, we get that
    \begin{enumerate}[label=(\roman*)]
        \setcounter{enumi}{2}
        \item $X$ is globally concave. \label{C-H3}
    \end{enumerate}
\end{thm}

\begin{proof}
    \ref{C-H1}: As the generalized Avez-Seifert Theorem \cite[Thm.\ 3.30]{LLS} is proved only for Lorentzian length spaces, we give a short argument for the existence of such geodesics invoking Proposition \ref{geodhomotop}: First, by timelike path-connectedness, there exists a timelike curve $\gamma$ between $x$ and $y$. Now defining the constant homotopy $H_0(s,t):=\gamma(t)$, Proposition \ref{geodhomotop} implies that $H_0$ is timelike homotopic to a homotopy $H_1$ of $\gamma$, in which all intermediate curves $H_1(s,.)$ are geodesics with the same endpoints as $\gamma$. In particular, $x$ and $y$ can be joined by a timelike geodesic. 
    
    To show uniqueness, assume that there were two distinct timelike geodesics $\gamma_0$ and $\gamma_1$ connecting the same endpoints. By \autoref{geodhomotop}, there exists a timelike path-homotopy $\gamma_s$ between $\gamma_0$ and $\gamma_1$ such that all $\gamma_s$ are geodesics from $x$ to $y$. However, by \autoref{großesThm}, each of these curves possesses a neighborhood in which it is the only geodesic from $x$ to $y$, a contradiction.

    \ref{C-H2}: By the first item, $X$ is uniquely geodesic. Hence, the claim follows from \autoref{contvargeod}.

    \ref{C-H3}: This immediately follows from \autoref{globalconcav}, again using that $X$ is uniquely geodesic by \ref{C-H1}.
\end{proof}

Finally, we will give an application of the previous theorem to globalize curvature bounds from above in the sense of triangle comparison. The classical way to globalize curvature bounds from above requires existence and uniqueness of geodesics, as well as continuity of the geodesic map (cf.\ \cite[2.4.9 Prop.]{bridson-haefliger} for the metric case and \cite[Thm.\ 4.6]{Patchwork} for the Lorentzian case). Since this is exactly the conclusion of \autoref{C-H}, we are able to prove the following: 

\begin{cor}
    Let $X$ be a globally hyperbolic, strongly causal, and locally causally closed regular Lorentzian pre-length space such that $\tau$ is continuous. If $X$ has timelike curvature bounded from above by zero and is future one-connected, then $X$ satisfies a global curvature bound from above by zero. 
\end{cor}

\begin{proof}
    As the curvature bound implies local concavity, cf.\ \autoref{felixprop}, we can invoke \autoref{C-H} to conclude that there exist unique timelike geodesics between timelike related points and these geodesics vary continuously. Hence, all assumptions of \cite[Thm.\ 4.6]{Patchwork} are satisfied (note that their assumption of $X$ being non-timelike locally isolating is implied by strong causality and timelike path-connectedness), which immediately yields that $X$ has global curvature bounded above by zero.
\end{proof}

\appendix
\section{Appendix}
\begin{lem}\label{problem}
    Let $A_0<A_1<\dots<A_N$ be real numbers. We define sequences by recursion as follows:
    \begin{align*}
        &A_i^0=A_i,\hspace{0.5cm}\text{ for all } 0\leq i\leq N,\\
        &A_0^k=A_0,\,A_N^k=A_N,\,A_i^k=\frac{A_{i-1}^{k-1}+A_{i+1}^{k-1}}{2},\hspace{0.5cm}\text{ for all }1\leq i\leq N-1, \, k \ge 1.
    \end{align*}
    Then $A_i^k\to \frac{iA_N+(N-i)A_0}{N}$ as $k \rightarrow \infty$.
\end{lem}

\begin{proof}
    Let $(X_i^k)_{k\in\mathbb N}$ be a Markov Chain on $(A_0,\ldots,A_N)$ given by the transition matrix 
    \[
    P =
    \begin{bmatrix}
    1 & 0 & 0 & 0 & \dots & 0 & 0 & 0  \\
    \frac{1}{2} & 0 & \frac{1}{2} & 0 & \dots & 0 & 0 & 0  \\
    0 & \frac{1}{2} & 0 & \frac{1}{2} & \dots & 0 & 0 & 0  \\
    \vdots & \vdots & \vdots & \vdots &  & \vdots & \vdots & \vdots \\
    0 & 0 & 0 & 0 & \dots & \frac{1}{2} & 0 & \frac{1}{2} \\
    0 & 0 & 0 & 0 & \dots & 0 & 0 &1
    \end{bmatrix}
    \]
    and the inital distribution $\delta_{A_i}$. This is essentially a symmetric random walk with absorbing barriers $A_0$ and $A_N$. We clearly have $A_i^k=\mathbb{E}(X_i^k)$ and the probability that the random walk hits one of the barriers is $1$. However, it is not immediately clear what the probability of it hitting $A_N$ before $A_0$ is. We call this probability $P_i$. Then $P_0=0$ and $P_N=1$, while for $2\leq i\leq N$ we have $P_i=\frac{P_{i-1}+P_{i+1}}{2}$. This final expression is equivalent to 
    \begin{equation*}
    P_{i+1}-P_i=P_i-P_{i-1}=:\Delta.
    \end{equation*}
    We can now write
    \begin{equation*}
        1=P_N=P_N-P_{N-1}+P_{N-1}-P_{N-2}\pm\dots-P_0+P_0=N\Delta
    \end{equation*}
    This implies that $\Delta=1/N$ and $P_i=i/N$. Consequently we get $\mathbb{P}(X_i^k=A_N)\to \frac{i}{N}$, $\mathbb{P}(X_i^k=A_0)\to\frac{N-i}{N}$ and for $2\leq j\leq N$, we have $\mathbb{P}(X_i^k=A_j)\to 0$. Hence, finally,
    \begin{equation*}
        \lim_{k\to\infty}A_i^k=\lim_{k\to\infty}\mathbb{E}(X_i^k)=\sum_{j=0}^N \lim_{k\to\infty}\mathbb{P}(X_i^k=A_j)A_j=\frac{iA_N+(N-i)A_0}{N}.\qedhere
    \end{equation*} 
\end{proof}

\addcontentsline{toc}{section}{Acknowledgements}
\section*{Acknowledgements}
We would like to thank Tobias Beran and Felix Rott for their feedback and valuable input at various stages of this project. This research was supported by the Austrian Science Fund (FWF) [Grants DOI \href{https://doi.org/10.55776/PAT1996423}{10.55776/PAT1996423} and \href{https://doi.org/10.55776/EFP6}{10.55776/EFP6}]. For open access purposes, the authors have applied a CC BY public copyright license to any author-accepted manuscript version arising from this submission.

\addcontentsline{toc}{section}{References}
\bibliography{Synthetic_Lorentzian_Cartan_Hadamard} 

@article {AlexanderBishop,
    AUTHOR = {Alexander, Stephanie B. and Bishop, Richard L.},
     TITLE = {The {H}adamard-{C}artan theorem in locally convex metric
              spaces},
   JOURNAL = {Enseign. Math. (2)},
  FJOURNAL = {L'Enseignement Math\'ematique. Revue Internationale. 2e
              S\'erie},
    VOLUME = {36},
      YEAR = {1990},
    NUMBER = {3-4},
     PAGES = {309--320},
      ISSN = {0013-8584},
   MRCLASS = {53C70 (53C22)},
  MRNUMBER = {1096422},
MRREVIEWER = {Werner\ Ballmann},
}

@book {BallmannLectures,
    AUTHOR = {Ballmann, Werner},
     TITLE = {Lectures on spaces of nonpositive curvature},
    SERIES = {DMV Seminar},
    VOLUME = {25},
      NOTE = {With an appendix by Misha Brin},
 PUBLISHER = {Birkh\"auser Verlag, Basel},
      YEAR = {1995},
     PAGES = {viii+112},
      ISBN = {3-7643-5242-6},
   MRCLASS = {53C21 (58F17)},
  MRNUMBER = {1377265},
MRREVIEWER = {Boris\ Hasselblatt},
       DOI = {10.1007/978-3-0348-9240-7},
       URL = {https://doi.org/10.1007/978-3-0348-9240-7},
}

@incollection {BallmannPaper,
    AUTHOR = {Ballmann, Werner},
     TITLE = {Singular spaces of nonpositive curvature},
 BOOKTITLE = {Sur les groupes hyperboliques d'apr\`es {M}ikhael {G}romov
              ({B}ern, 1988)},
    SERIES = {Progr. Math.},
    VOLUME = {83},
     PAGES = {189--201},
 PUBLISHER = {Birkh\"auser Boston, Boston, MA},
      YEAR = {1990},
      ISBN = {0-8176-3508-4},
   MRCLASS = {53C23 (53C22)},
  MRNUMBER = {1086658},
       DOI = {10.1007/978-1-4684-9167-8\_10},
       URL = {https://doi.org/10.1007/978-1-4684-9167-8_10},
}

@book {BeemEhrlich,
    AUTHOR = {Beem, John K. and Ehrlich, Paul E. and Easley, Kevin L.},
     TITLE = {Global {L}orentzian geometry},
    SERIES = {Monographs and Textbooks in Pure and Applied Mathematics},
    VOLUME = {202},
   EDITION = {Second},
 PUBLISHER = {Marcel Dekker, Inc., New York},
      YEAR = {1996},
     PAGES = {xiv+635},
      ISBN = {0-8247-9324-2},
   MRCLASS = {53C50 (53-02 83-02)},
  MRNUMBER = {1384756},
MRREVIEWER = {Peter\ R.\ Law},
}

@article {curvature,
    AUTHOR = {Beran, Tobias and Kunzinger, Michael and Rott, Felix},
     TITLE = {On curvature bounds in {L}orentzian length spaces},
   JOURNAL = {J. Lond. Math. Soc. (2)},
  FJOURNAL = {Journal of the London Mathematical Society. Second Series},
    VOLUME = {110},
      YEAR = {2024},
    NUMBER = {2},
     PAGES = {Paper No. e12971, 41},
      ISSN = {0024-6107,1469-7750},
   MRCLASS = {53C23 (53B30 53C50)},
  MRNUMBER = {4781260},
       DOI = {10.1112/jlms.12971},
       URL = {https://doi.org/10.1112/jlms.12971},
}

@article {Patchwork,
    AUTHOR = {Beran, Tobias and Napper, Lewis and Rott, Felix},
     TITLE = {Alexandrov's patchwork and the {B}onnet-{M}yers theorem for
              {L}orentzian length spaces},
   JOURNAL = {Trans. Amer. Math. Soc.},
  FJOURNAL = {Transactions of the American Mathematical Society},
    VOLUME = {378},
      YEAR = {2025},
    NUMBER = {4},
     PAGES = {2713--2743},
      ISSN = {0002-9947,1088-6850},
   MRCLASS = {53C50 (51K10 53B30 53C23)},
  MRNUMBER = {4880460},
       DOI = {10.1090/tran/9372},
       URL = {https://doi.org/10.1090/tran/9372},
}

@book {bridson-haefliger,
    AUTHOR = {Bridson, Martin R. and Haefliger, Andr\'e},
     TITLE = {Metric spaces of non-positive curvature},
    SERIES = {Grundlehren der mathematischen Wissenschaften [Fundamental
              Principles of Mathematical Sciences]},
    VOLUME = {319},
 PUBLISHER = {Springer-Verlag, Berlin},
      YEAR = {1999},
     PAGES = {xxii+643},
      ISBN = {3-540-64324-9},
   MRCLASS = {53C23 (20F65 53C70 57M07)},
  MRNUMBER = {1744486},
MRREVIEWER = {Athanase\ Papadopoulos},
       DOI = {10.1007/978-3-662-12494-9},
       URL = {https://doi.org/10.1007/978-3-662-12494-9},
}

@article {LytchakIvanov,
    AUTHOR = {Ivanov, Sergei and Lytchak, Alexander},
     TITLE = {Rigidity of {B}usemann convex {F}insler metrics},
   JOURNAL = {Comment. Math. Helv.},
  FJOURNAL = {Commentarii Mathematici Helvetici. A Journal of the Swiss
              Mathematical Society},
    VOLUME = {94},
      YEAR = {2019},
    NUMBER = {4},
     PAGES = {855--868},
      ISSN = {0010-2571,1420-8946},
   MRCLASS = {53B40 (53C23 53C60)},
  MRNUMBER = {4046007},
MRREVIEWER = {Wei\ Zhao},
       DOI = {10.4171/cmh/476},
       URL = {https://doi.org/10.4171/cmh/476},
}

@article {LLS,
    AUTHOR = {Kunzinger, Michael and S\"amann, Clemens},
     TITLE = {Lorentzian length spaces},
   JOURNAL = {Ann. Global Anal. Geom.},
  FJOURNAL = {Annals of Global Analysis and Geometry},
    VOLUME = {54},
      YEAR = {2018},
    NUMBER = {3},
     PAGES = {399--447},
      ISSN = {0232-704X,1572-9060},
   MRCLASS = {53C23 (53B30 53C50 53C80)},
  MRNUMBER = {3867652},
MRREVIEWER = {Benjam\'in\ Olea},
       DOI = {10.1007/s10455-018-9633-1},
       URL = {https://doi.org/10.1007/s10455-018-9633-1},
}

@book {oneil,
    AUTHOR = {O'Neill, Barrett},
     TITLE = {Semi-{R}iemannian geometry},
    SERIES = {Pure and Applied Mathematics},
    VOLUME = {103},
      NOTE = {With applications to relativity},
 PUBLISHER = {Academic Press, Inc. [Harcourt Brace Jovanovich, Publishers],
              New York},
      YEAR = {1983},
     PAGES = {xiii+468},
      ISBN = {0-12-526740-1},
   MRCLASS = {53-01 (53B30 53C50 83-02)},
  MRNUMBER = {719023},
MRREVIEWER = {N.\ V.\ Mitskevich},
}

@article {MinguzziSuhr,
    AUTHOR = {Minguzzi, E. and Suhr, S.},
     TITLE = {Lorentzian metric spaces and their {G}romov-{H}ausdorff
              convergence},
   JOURNAL = {Lett. Math. Phys.},
  FJOURNAL = {Letters in Mathematical Physics},
    VOLUME = {114},
      YEAR = {2024},
    NUMBER = {3},
     PAGES = {Paper No. 73, 63},
      ISSN = {0377-9017,1573-0530},
   MRCLASS = {53C50 (51F99 83C45)},
  MRNUMBER = {4752400},
MRREVIEWER = {S.\ M. B. Kashani},
       DOI = {10.1007/s11005-024-01813-z},
       URL = {https://doi.org/10.1007/s11005-024-01813-z},
}

@article{BraunMcCann,
    AUTHOR = {Braun, M. and McCann, R. J.},
    TITLE = {Causal convergence conditions through variable timelike Ricci curvature bounds},
    YEAR = {2023},
    DOI = {10.48550/arXiv.2312.17158},
    URL = {https://arxiv.org/abs/2312.17158}
    
}

@article{Olaf,
    AUTHOR = {M{\"u}ller, O.},
    TITLE = {Functors in {L}orentzian geometry -- three variations on a theme},
    YEAR = {2022},
    DOI = {10.48550/arXiv.2205.01617},
    URL = {https://arxiv.org/pdf/2205.01617}
    
}

@inproceedings {GromovI,
 AUTHOR = {Gromov, M.},
 TITLE = {Hyperbolic manifolds, groups and actions},
 BOOKTITLE = {Riemann surfaces and related topics: {P}roceedings of the 1978
 {S}tony {B}rook {C}onference ({S}tate {U}niv. {N}ew {Y}ork,
 {S}tony {B}rook, {N}.{Y}., 1978)},
 SERIES = {Ann. of Math. Stud., No. 97},
 PAGES = {183--213},
 PUBLISHER = {Princeton Univ. Press, Princeton, NJ},
 YEAR = {1981},
 MRCLASS = {53C15 (53C45 58F17)},
 MRNUMBER = {624814},
MRREVIEWER = {Elmer G. Rees},
}

@incollection {GromovII,
 AUTHOR = {Gromov, M.},
 TITLE = {Hyperbolic groups},
 BOOKTITLE = {Essays in group theory},
 SERIES = {Math. Sci. Res. Inst. Publ.},
 VOLUME = {8},
 PAGES = {75--263},
 PUBLISHER = {Springer, New York},
 YEAR = {1987},
 MRCLASS = {20F32 (20F06 20F10 22E40 53C20 57R75 58F17)},
 MRNUMBER = {919829},
MRREVIEWER = {Christopher W. Stark},
 DOI = {10.1007/978-1-4613-9586-7\_3},
 URL = {https://doi.org/10.1007/978-1-4613-9586-7_3},
}
\bibliographystyle{aomalpha}
\end{document}